\documentclass[a4paper, 12pt]{amsart} 

\usepackage{amssymb, amsmath, amsthm}
\usepackage{amscd}
\usepackage{mathrsfs}
\usepackage{lscape}
\usepackage[all]{xy}

\theoremstyle{plain}
\newtheorem{theorem}{Theorem}[section]
\newtheorem{lemma}[theorem]{Lemma}
\newtheorem{corollary}[theorem]{Corollary}
\newtheorem{definition}[theorem]{Definition}
\newtheorem{proposition}[theorem]{Proposition}
\newtheorem{remark}[theorem]{Remark}

\newtheorem{example}[theorem]{Example}

\makeatletter

 \@addtoreset{equation}{section}
\makeatother

\renewcommand{\L}[0]{\mathrm{L}}
\newcommand{\Lip}[0]{\mathrm{Lip}}
\newcommand{\N}[0]{\mathbf N}
\newcommand{\jump}[1]{\ensuremath{[\![#1]\!]} }
\newcommand{\im}[0]{\mathrm{im}}
\newcommand{\cpt}[0]{\mathrm{c}}
\newcommand{\T}[0]{\triangle}
\newcommand{\A}[0]{\mathfrak A}
\newcommand{\Lipb}[0]{\mathrm{Lip_b}}
\newcommand{\BT}[0]{\mathrm{Lip_b^{BT}}}
\newcommand{\MT}[0]{\mathrm{Lip_b^{MT}}}
\newcommand{\CO}[0]{\mathrm{CO}}

\title
[Current and measure homologies]
{The coincidence of the current homology and the measure homology via a new topology on spaces of Lipschitz maps}
\author{Ayato Mitsuishi}
\address
{Mathematical Institute, Tohoku University, Sendai 980-8578, JAPAN}
\email[A.~Mitsuishi]{mitsuishi@math.tohoku.ac.jp}

\begin{document}
\maketitle

\begin{abstract}
We consider the category of all locally Lipschitz contractible metric spaces and all locally Lipschitz maps, which is a wide class of metric spaces, including all finite dimensional Alexandrov spaces and all CAT spaces. 
We also consider the chain complex of normal currents with compact support in a metric space in the sense of Ambrosio and Kirchheim. 
In the present paper, its homology is proved to be a homotopy invariant on the category. 

To prove this result, we define a new topology on a space of Lipschitz maps between arbitrary metric spaces.  
This topology is proved to coincide with the usual $C^1$-topology on the space of $C^1$-maps between compact Riemannian manifolds. 
\end{abstract}

\section{Introduction} \label{sec:intro}
Currents in smooth manifolds were introduced by de Rham (\cite{DR}), which were defined to be continuous linear functionals on the spaces of smooth forms. 
Currents in metric spaces (called metric currents) were introduced by Ambrosio and Kirchheim (\cite{AK}), by providing a space of formal forms on metric spaces (see also \cite{L}). 
When we will say just currents, they will indicate metric currents. 
The space of all normal currents of compact support in a metric space (and in a smooth manifold) is known to be a chain complex.
We call its homology the {\it current homology}. 
It is a group depending on the metric structure. 
Here, a natural question arise: 
is it a topological invariant? 
In the present paper, we give an answer to this probelm. 


To state our results, let us fix notation.
For a metric space $X$, we denote by 
$\mathbf H_\ast(X)$ the current homology, that is the homology of normal currents with compact support. 
We also consider the {\it measure homology} $\mathscr H_\ast(X)$ introduced by Thurston \cite{T}.
Thurston originally defined his measure homology only for smooth manifolds. 
In the paper, we use a slightly modified version of the measure homology which can be defined for any topological space, 
considered by Hansen (\cite{H}) and Zastrow (\cite{Z}). 

Let us consider metric spaces satisfying the following condition. 
\begin{definition}[\cite{Y}, \cite{MY}, cf. \cite{Mi}] \label{def:LLC} \upshape
A metric space $X$ is said to be {\it locally Lipschitz contractible}, 
shortly LLC, if for every $x \in X$ and open neighborhood $O$ of $x$, there exist an open set $O'$ of $X$ with $x \in O' \subset O$, a Lipschitz map $h : O' \times [0,1] \to O$ and a point $y \in O$ such that $h_0$ is the inclusion $O' \hookrightarrow O$ and $h_1$ is the constant map of value $y$, where $h_t$ denotes the map $h(\cdot,t)$ for $t \in [0,1]$.
\end{definition}

\begin{example} \upshape \label{ex:LLC}
There are many objects being locally Lipschitz contractible in metric geometry. 
For instance, normed spaces, Riemannian manifolds, locally CAT-spaces, and finite dimensional locally Alexandrov spaces, satisfy this property.
Here, locally CAT spaces (resp. locally Alexandrov spaces) are length metric spaces of curvature bounded locally from above (resp. below) in the sense of Alexandrov. 
For their precise definitions and fundamental properties, we refer to \cite{BGP} and \cite{BBI}.
Among these examples, it is not trivial that finite dimensional locally Alexandrov spaces are LLC. 
This fact was proved in \cite{MY}. 
Further, $C_k$-spaces and $C_L$-spaces in the sense of Ohta \cite{O} are also LLC. 
All spaces appeared above are actually known to satisfy a property stronger than the LLC-condition, called the strong local Lipschitz contractibility. 
For the definition of this property, see \cite{MY}.
For other properties similar to the LLC-condition and their relation, we refer \cite{Mi} and its reference.

For an LLC metric space $X$, if $O$ is an open subset of it, then $O$ is also LLC, and if another metric space $Y$ is locally bi-Lipschitz homeomorphic to $X$, then $Y$ is LLC (\cite{Mi}).
\end{example}

Using above terminologies, 
a main result is stated as follows. 
\begin{theorem} \label{thm:main thm0}
On the category of all LLC metric spaces and all locally Lipschitz maps, there is a natural isomorphism $\mathbf H_\ast \cong \mathscr H_\ast$ between the current and measure homologies.

That is, for each LLC metric space $X$, there is an isomorphism $\eta_X : \mathbf H_\ast(X) \to \mathscr H_\ast(X)$ such that for any locally Lipscihtz map $f : X \to Y$ to an LLC metric space $Y$, we have $\mathscr H_\ast(f) \circ \eta_X = \eta_Y \circ \mathbf H_\ast(f)$.
\end{theorem}
A relative version of Theorem \ref{thm:main thm0} also holds (Theorem \ref{thm:main thm:rel}).

We should remark that the measure homology and the current homology were proved to satisfy the axiom of Eilenberg and Steenrod in \cite{H}, \cite{Z} and \cite{Mo}. 
However, there is an LLC metric space which does not have the homotopy type of CW-complex (\cite{Mi}). 
Hence, we can not apply the uniqueness of homology thoery to prove Theorem \ref{thm:main thm0}.
We will use the cosheaf theory to prove Theorem \ref{thm:main thm0}.

From the fundamental properties of currents and the measure homology, together with Theorem \ref{thm:main thm0}, we obtain the following corollaries. 

\begin{corollary} \label{cor:top inv}
The functor $\mathbf H_\ast$ can be extended to the category of all LLC metric spaces and all {\rm continuous} maps, which is natrually isomorphic to $\mathscr H_\ast$.
In partiular, if two LLC metric spaces $X$ and $Y$ are homotopic, then the groups $\mathbf H_\ast(X)$ and $\mathbf H_\ast(Y)$ are isomorphic.
\end{corollary}

\begin{corollary} \label{cor:Haus dim}
Let $X$ be an locally Lipschitz contractible metric space of Hausdorff dimension less than $n$, for $n \in \mathbb Z_{\ge 0}$. 
Then, its measure homology $\mathscr H_k(X)$ is trivial for every $k \ge n$.
\end{corollary}

By \cite{H}, \cite{Z} 
and Theorem \ref{thm:main thm0}, we obtain 

\begin{corollary} \label{cor:sh}
If an LLC metric space $X$ has the homotopy type of a CW-complex, then the current homology is isomorphic to the singular homology of real coefficient. 

In particular, if $X$ is homotopic to a finite CW-complex, then $\mathbf H_k(X)$ are finite dimensional for all $k \ge 0$.
\end{corollary}

\begin{corollary} \label{cor:alex}
Let $X$ be an $n$-dimensional compact Alexandrov space without boundary. 
Suppose that the manifold part of $X$ is orientable. 
Then, $\mathbf H_n(X)$ and $\mathscr H_n(X)$ are isomorphic to $\mathbb R$.
\end{corollary}

In the course of the proof of Theorem \ref{thm:main thm0}, we need a new topology on a space of Lipschitz maps between arbitrary metric spaces. 
For metric spaces $Z$ and $X$, we denote by $\Lipb(Z,X)$ the space of all Lipschitz maps whose image is bounded. 
\begin{theorem} \label{thm:top Lip}
For any metric spaces $Z$ and $X$, there is a topology on $\Lipb(Z,X)$ satisfying the following.
\begin{itemize}
\item[(0)]
The topology on $\Lipb(Z,X)$ is metrizable and finner than the topology induced by the supremum distance.
\item[(1)] 
If a sequence $f_j$ converges to $f$ in the topology on $\Lipb(Z,X)$, then $\sup_j \Lip(f_j) < \infty$.
\item[(2)]
Let $\phi : X \to Y$ be a Lipschitz map between metric spaces. 
Then, a map $\phi_\# : \Lipb(Z,X) \to \Lipb(Z,Y)$ given by $\sigma \mapsto \phi \circ \sigma$ is continuous.
Further, if $X$ is a subset of $Y$ and $\phi$ is the inclusion $X \hookrightarrow Y$, then $\phi_\#$ is a topological embedding.
\item[(3)]
Let $\psi : W \to Z$ be a Lipschitz map between metric spaces. 
Then, a map $\psi^\# : \Lipb(Z,X) \to \Lipb(W,X)$ given by $\sigma \mapsto \sigma \circ \psi$ is continuous.
\item[(4)]
A canonical map 
\begin{align*}
\Lipb(Z,X) \times \Lipb(W,Y) \to \Lipb(Z \times W, X \times Y) 
\end{align*}
defined by $(\phi, \psi) \mapsto \phi \times \psi$ is continuous, where $\phi \times \psi$ is given by $(\phi \times \psi)(z,w) = (\phi(z),\psi(w))$ for $(\phi, \psi) \in \Lipb(Z,X) \times \Lipb(W,Y)$ and $(z,w) \in Z \times W$. 
\item[(5)] Let $\{\ast\}$ denote a single-point set. 
Then, a canonical map 
\[
\Lipb(\{\ast\}, X) \ni f \mapsto f(\ast) \in X
\]
is homeomorphic.
\end{itemize}
\end{theorem}

The property (1) of a topology in Theorem \ref{thm:top Lip} is a crucial difference from the compact-open topology (see Remark \ref{rem:01}).
The existence of such a topology will be discussed in \S \ref{sec:top Lip}.

\subsection{More details on Theorem \ref{thm:main thm0}} \label{subsec:presice statement}
Let us first give an intuitive explanation of a coincidence of currents and measures in the $0$-th chain groups.
By the definition, $0$-currents are continuous functionals on the space of real-valued bounded Lipschitz functions on a metric space $X$.
They are like Schwart's distributions.
On the other hands, a signed Borel measure on $X$ can be actually regarded as a distribution, 
by the integration of functions with respect to the measure. 
Thus, we have an intuitive identification:
\begin{equation} \label{eq:correspondence}
\{\text{$0$-currents in $X$}\} 
= \{\text{signed Borel measures on $X$}\}. 
\end{equation}
We will verify that the above equality actually has the meaning (Lemma \ref{lem:0-th}). 
This is considered as a higher dimensional analogue of the correspondence \eqref{eq:correspondence}. 

Let $\mathscr C_\bullet(X)$ denote the measure chain complex and $\N_\bullet^\cpt(X)$ denote the chain complex of normal currents with compact support. 
Thier homologies are no longer than $\mathscr H_\ast(X)$ and $\mathbf H_\ast(X)$. 
We will define some chain complex $\mathscr C_\bullet^\L (X)$ with chain maps 
\begin{align} 
&\mathscr C_\bullet^\L(X) \to \mathscr C_\bullet(X), \label{eq:incl} \\
&\mathscr C_\bullet^\L(X) \to \mathbf N_\bullet^\cpt(X). \label{eq:map}
\end{align}
Here, to prove that the map \eqref{eq:map} is well posed, we need a new topology stated as in Theorem \ref{thm:top Lip} (see Remarks \ref{rem:not conti} and \ref{rem:not integrable}). 

Theorem \ref{thm:main thm0} can be stated more precisely as follow. 
\begin{theorem} \label{thm:main thm}
On the category of all metric spaces and all locally Lipschitz maps, the correspondence \eqref{eq:incl} and \eqref{eq:map} are natural transformations between covariant functors $\mathscr C_\bullet$, $\mathscr C_\bullet^\L$ and $\mathbf N_\bullet$ to the category of all chain complices of real vector spaces and all chain maps. 

Further, if the natural transformations \eqref{eq:incl} and \eqref{eq:map} are restricted to the full subcategory consisting of all locally Lipschitz contractible metric spaces, then the corresponding induced maps between the homologies 
\[
\mathscr H_\ast \leftarrow \mathscr H_\ast^\L \to \mathbf H_\ast
\]
are isomorphic.
Here, $\mathscr H_\ast^\L$ denotes the homology of $\mathscr C_\bullet^\L$.
\end{theorem}

\vspace{10pt} \noindent {\bf Organization}. 
The organization of this paper is as follows.
In \S \ref{sec:current}, we recall the definition of metric currents in the sense of Ambrosio and Kirchheim and define the current homology of metric spaces. 
In \S \ref{sec:mh}, we recall the definition of the measure homology. 
We also define a chain complex $\mathscr C_\bullet^\L$ mentioned in \S \ref{subsec:presice statement}, that will be called the Lipschitz measure chian complex, of a metric space. 
Using Theorem \ref{thm:top Lip}, we define chain maps from the Lipschitz measure chain complex to the measure and current complicies desired as \eqref{eq:incl} and \eqref{eq:map}. 
In \S \ref{sec:proof:main thm}, we recall the notion of cosheaf and its fundamental property, and a way to show a coincidence of two homologies associated to cosheaves (Theorem \ref{thm:coincidence}).
We prove that the functor taking the measure (resp. Lipschitz measure) chain complex on each open set of a topological (resp. metric) space is a cosheaf.
Verifying that these cosheaves satisfy the assumption of Theorem \ref{thm:coincidence}, we give proofs of Theorem \ref{thm:main thm} and Corollaries \ref{cor:top inv}--\ref{cor:alex}.
In \S \ref{sec:top Lip}, we prove Theorem \ref{thm:top Lip}, that is, we give a reasonable topology on the set of all bounded Lipschitz maps between metric spaces. 

\vspace{10pt}
\noindent
{\bf Acknowledgment}. 
The author expresses his thanks to Professor Takamitsu Yamauchi for discussions about a topology of function spaces that will appear in Remark \ref{rem:Lang top}.
He thanks Yu Kitabeppu for comments in a construction of an isometric embedding of a metric space into a Banach space that will be used in \S \ref{subsec:double dual}.
This work was supported by Research Fellowships of the Japan Society for the Promotion of Science for Young Scientists.

\section{Currents and its homologies} \label{sec:current}

Let us recall the definition of metric currents in the sense of Ambrosio and Kirchheim (\cite{AK}). 

Let $X$ denote a metric space. 
Let $\Lip(X)$ be the set of all Lipschitz functions from $X$ to $\mathbb R$ and $\Lipb(X)$ the subset of $\Lip(X)$ consisting of bounded functions.
For a map $f$ between metric spaces, we denote by $\Lip(f)$ its Lipschitz constant.
Let $k$ denote a nonnegative integer.
The space $\mathcal D^k(X) := \Lipb(X) \times \left(\Lip(X)\right)^k$ is considered as the space of $k$-forms on $X$. 
An element $(f,\pi_1,\dots,\pi_k) \in \mathcal D^k(X)$ is written as $f d \pi_1 \wedge \dots \wedge d \pi_k$ or $f d \pi$ for shortly.

\begin{definition}[\cite{AK}] \label{def:current} \upshape
A $k$-{\it current} in $X$ is a multilinear map 
\[
T : \mathcal D^k(X) \to \mathbb R
\]
such that it satisfies the following three axioms: 
\begin{itemize}
\item[(locality)] 
for $f d \pi \in \mathcal D^k(X)$, we have $T(f d \pi) = 0$ whenever $\pi_i$ is constant on $\{f \neq 0\}$ for some $i$; 
\item[(continuity)] 
if a sequence $\pi^h = (\pi_i^h) \in \left(\Lip(X)\right)^k$, $h \in \mathbb N$, converges to $\pi = (\pi_i)$ pointwise as $h \to \infty$ with $\sup_{i,h} \Lip(\pi_i^h) < \infty$, then we have 
\[
T(f d \pi) = \lim_{h \to \infty} T(f d \pi^h)
\]
for every $f \in \Lipb(X)$; 
\item[(finite mass)]
there is a finite tight Borel measure $\mu$ on $X$ satisfying 
\[
|T(f d \pi)| \le \prod_{i=1}^k \Lip(\pi_i) \int_X |f| \, d \mu
\]
for all $f d \pi \in \mathcal D^k(X)$.
\end{itemize}
\end{definition}
Let $T$ denote a $k$-current in $X$.
The {\it support} of $T$ is defined by the intersection of all supports of $\mu$ satisfying the finite mass axiom for $T$. 
We will deal with only currents of compact support in the paper. 
The {\it boundary} of $T$ is a multi-linear map $\partial T : \mathcal D^{k-1}(X) \to \mathbb R$ defined by 
\[
\partial T(f d \pi) = T(d f \wedge d \pi).
\]
By the locality, $\partial \partial T = 0$ holds.
The boundary $\partial T$ satisfies the continuity and locality.
If $\partial T$ has finite mass, then $T$ is said to be {\it normal}.
The set of all normal $k$-currents in $X$ is denoted by $\N_k(X)$. 
Thus, $\N_\bullet(X)$ becomes a chain complex.
Since supports being compact preserves under the boundary, the space  $\N_\bullet^\cpt(X)$ of all compactly supported normal currents is also a chain complex.
In the present paper, we consider its homology, denoted by 
\[
\mathbf H_\ast(X) := H_\ast(\N_\bullet^\cpt(X)).
\] 
We call it the {\it current homology}.
For the empty-set, we set $\N_\bullet^\cpt(\emptyset) = 0$. 

For another metric space $Y$ with a locally Lipschitz map $\phi : X \to Y$, we have a chain map
\[
\phi_\# : \N_\bullet^\cpt(X) \to \N_\bullet^\cpt(Y)
\]
defined by 
\[
\phi_\# T (f d \pi) = T (f \circ \phi\, d (\pi \circ \phi))
\]
for all $T \in \N_k^\cpt(X)$, $f d \pi \in \mathcal D^k(X)$ and $k \ge 0$.
This is actually defined, since normal currents are compactly supported (see \cite{Mi}).
Thus, the chain complex $\mathbf N_\bullet^\cpt$ is a covariant functor from the category of all metric spaces and all locally Lipschitz maps to the category of all chain complices and all chain maps. 
The current homology $\mathbf H_\ast$ is a covariant functor to the category of all vector spaces and all linear maps.

For a metric space $X$ and its subset $A$, since the inclusion $A \hookrightarrow X$ induces an injective chain map  $\N_\bullet^\cpt(A) \to \N_\bullet^\cpt(X)$, we can regard $\N_\bullet^\cpt(A)$ as a subcomplex of $\N_\bullet^\cpt(X)$.
We set 
\[
\N_\bullet^\cpt(X,A) := \N_\bullet^\cpt(X) / \N_\bullet^\cpt(A).
\]
Its homology is denoted by 
\[
\mathbf H_\ast(X,A)
\]
called the current homology of $(X,A)$.
A map $f : (X,A) \to (Y,B)$ between pairs of metric spaces is said to be locally Lipschitz if so is $f : X \to Y$. 
If $f : (X,A) \to (Y,B)$ is a locally Lipschitz map, then a chain map $f_\# : \N_\bullet^\cpt(X,A) \to \N_\bullet^\cpt(Y,B)$ and a linear map $f_\bullet : \mathbf H_\ast(X,A) \to \mathbf H_\ast(Y,B)$ are induced.

\section{(Lipschitz) Measure homology} \label{sec:mh}
In this section, we recall the definition of measure chain complex of topological spaces. 
We introduce the Lipschitz measure chain complex of metric spaces with chain maps from it to the masure chain complex and the complex of currents mentioned as in \eqref{eq:incl} and \eqref{eq:map}. 

\subsection{Fixing terminology from measure theory}
Before defining the measure homology, let us fix the terminology and notation about measures. 
Let $(T,\mathscr A)$ be a measurable space. 
We say that a function $\mu : \mathscr A \to \mathbb R \cup \{\infty, - \infty\}$ is a signed measure if $\mu(\emptyset) = 0$, 
the image of it does not contains both values $\infty$ and $-\infty$, 
and it is $\sigma$-additive.
A subset $D$ of $T$ which is not necessarily measurable is called a {\it determination set} of a signed measure $\mu$ on $(T,\mathscr A)$ if every measurable set $A \in \mathscr A$ contained in $T - D$ is of zero measure in $\mu$.

Let $S$ be a topological space. 
Let us denote by $\mathscr M_\cpt(S)$ the real vector space of all singed Borel measures on $S$ of finite total variation having a compact determination set. 
For any continuous map $f : S \to S'$ between topological spaces, a linear map $f_\# : \mathscr M_\cpt(S) \to \mathscr M_\cpt(S')$ is given by sending $\mu \in \mathscr M_\cpt(S)$ to the push-forward measure $f_\# \mu = \mu (f^{-1}(\cdot)) \in \mathscr M_\cpt(S')$.
Obviously, if $f : S \to S'$ is a topological embedding, then $f_\#$ is injective.

For a Borel set $A$ of $S$ and a signed Borel measure $\mu$ on $S$, we define a signed Borel measure $\mu \lfloor A$ on $S$ by 
\[
\mu \lfloor A (B) = \mu(A \cap B) 
\]
for every Borel set $B$ of $S$. 
By the definition, it has a determination set $A$. 
Further, we use the same symbol $\mu \lfloor A$ meaning 
the restriction of $\mu$ to the Borel $\sigma$-algebra of $A$, which is a signed Borel measure on $A$.

\subsection{Measure homology} \label{subsec:mh}
In this subsection, let $X$ denote a topological space. 
For $k \ge 0$, we denote by $\T^k$ a regular $k$-simplex. 
Let us denote by $C(\T^k,X)$ the space of all singular $k$-simplices in $X$ with the compact-open topology.
Note that $X$ is Hausdorff if and only if so is $C(\T^k,X)$.
Recall that if $X$ is a metric space, then the compact-open topology on $C(\T^k,X)$ coincides with the topology induced from the uniform distance.

The measure $k$-th chain group of $X$ is defined by 
\[
\mathscr C_k(X) := \mathscr M_\cpt(C(\T^k,X)).
\]
For $i = 0,\dots,k$, the restriction $r_i : C(\T^k,X) \to C(\T^{k-1},X)$ to the $i$-th face of $\T^k$ is continuous in the compact-open topology. 
This induces a linear map $r_i {}_\# : \mathscr C_k(X) \to \mathscr C_{k-1}(X)$ by the push-forward of measures. 
Then, the following map given by 
\[
\partial = \sum_{i=0}^k (-1)^i r_i {}_\# : \mathscr C_k(X) \to \mathscr C_{k-1}(X)
\]
is easily verified to satisfy $\partial \partial = 0$.
So, $(\mathscr C_\bullet(X), \partial)$ becomes a chain complex and is called the {\it measure chain complex}.
Its homology is called the {\it measure homology} and is denoted by 
\[
\mathscr H_\ast(X). 
\]

Let $Y$ denote another topological space with $\phi : X \to Y$ a continuous map. 
Since the composition $\phi_\#: C(\T^k,X) \to C(\T^k,Y); \sigma \mapsto \phi \circ \sigma$ is continuous, 
it induces a chain map $\phi_\# : \mathscr C_\bullet(X) \to \mathscr C_\bullet(Y)$ by the push-forward of measures. 

Let $A$ be a subspace of $X$. 
The inclusion $A \hookrightarrow X$ induces a topological embedding $C(\T^k, A) \to C(\T^k,X)$ for every $k \ge 0$. 
Hence, it induces an injective chain map $\mathscr C_\bullet(A) \to \mathscr C_\bullet(X)$. 
Thus, we regard $\mathscr C_\bullet(A)$ as a subcomplex of $\mathscr C_\bullet(X)$. 
The quotient $\mathscr C_\bullet(X) / \mathscr C_\bullet(A)$ is denoted by
\[
\mathscr C_\bullet(X,A)
\]
called the measure chain complex of $(X,A)$. 
Its homology is denoted by 
\[
\mathscr H_\ast(X,A)
\]
called the measure homology of the pair $(X,A)$.
When $A = \emptyset$, we identify $\mathscr H_\ast(X,\emptyset)$ with $\mathscr H_\ast(X)$.

\subsection{Fundamental proprties of the topology in Theorem \ref{thm:top Lip}} \label{subsec:top Lip}
We give remarks about a topology stated in Theorem \ref{thm:top Lip} and prove fundamental properties of the topology. 

For metric spaces $Z$ and $X$, $C(Z,X)$ denotes the space of all continuous maps from $Z$ to $X$ and $\Lip(Z,X)$ denotes the space of all Lipscihtz maps from $Z$ to $X$. 
When $Z$ is compact, $\Lip(Z,X)=\Lipb(Z,X)$ as sets. 
In this case, we always consider that $\Lip(Z,X)$ has the topology given in Theorem \ref{thm:top Lip}.


\begin{remark} \label{rem:01} \upshape
Among properties of a topology as in Theorem \ref{thm:top Lip}, the property $(1)$ is a crucial difference from the compact-open topology.
Indeed, there is a sequence of real-valued Lipschitz functions on $[0,1]$ converging to a Lipschitz function uniformly such that the Lipschitz constants diverges to infinity.
Hence, any topology satisfying Theorem \ref{thm:top Lip} is strictly finner than the compact-open topology, in general.

Moreover, setting functions $f_t, f_0 \in \Lip([0,1], \mathbb R)$, where $1/2 \ge t > 0$, as 
\[
f_t (x) = \left\{
\begin{aligned}
&0 &&\text{ if } 0 \le x \le t \\
&(x - t) / \sqrt t &&\text{ if } t \le x \le 2 t \\
&\sqrt t &&\text{ if } 2 t \le x \le 1
\end{aligned}
\right.
\]
and $f_0(x) = 0$ everywhere, 
the set $K = \{f_t \mid 0 \le t \le 1/ 2\}$ is compact in the compact-open topology, however it is not compact in any topology satisfying $(0)$ and $(1)$ of Theorem \ref{thm:top Lip}.
\end{remark}

\begin{remark} \label{rem:Lang top} \upshape
In \cite{L}, Lang considered another topology on the space of all compactly supported Lipschitz real-valued functions on a (locally compact) metric space. 
We denote it by $\mathfrak T$ for the moment.
As stated there, one can prove that $\mathfrak T$ satisfies the following property \eqref{eq:Lang}. 
Let $f_j$ and $f$ be compactly supported Lipschitz functions with $j \in \mathbb N$. 
Then, we have 
\begin{equation} \label{eq:Lang}
\text{$f_j \to f$ in $\mathfrak T$} \iff 
\text{$\sup_j \Lip(f_j) < \infty$ and $f_j \to f$ uniformly.}
\end{equation}
In particular, the topology 
$\mathfrak T$ is coarser than any topology satisfying Theorem \ref{thm:top Lip}.
However, the author does not know whether $\mathfrak T$ is metrizable or not.

A topology satisfying Theorem \ref{thm:top Lip} which will be given in \S \ref{sec:top Lip} has a property stronger than $(1)$ of Theorem \ref{thm:top Lip}. 
It actually holds that the function $\Lip(\cdot)$ taking the smallest Lipschitz constant on $\Lipb(Z,X)$ is continuous (Proposition \ref{prop:met1}). 
Hence, even if a sequence $f_j$ converges to $f$ in $\Lipb(Z,X)$ uniformly and $\sup_j \Lip(f_j) < \infty$, 
$f_j$ may diverge in our topology, in general. Indeed, there are functions $f_j : [0,1] \to \mathbb R$ with $\Lip(f_j) = 1$ for $j \in \mathbb N$, such that $f_j$ converges to a constant function uniformly. 

Note that if there exists a topology on the space of all (bounded) Lipschitz maps between metric spaces such that it is metrizable and satisfies the property \eqref{eq:Lang}, then such a topology obviously satisfies the conclusion of Theorem \ref{thm:top Lip}. 
\end{remark}

The following statements are corollaries to Theorem \ref{thm:top Lip}.

\begin{corollary}
If $\phi : X \to Y$ and $\psi : Z \to W$ are bi-Lipschitz homeomorphisms, then 
the map $\Lipb(W,X) \ni f \mapsto \phi \circ f \circ \psi \in \Lipb(Z,Y)$ is homeomorphic.
\end{corollary}
\begin{proof}
This follows from the properties (2) and (3) in Theorem \ref{thm:top Lip}. 
\end{proof}

\begin{corollary}\label{cor:open emb}
Let $U$ be an open set in a metric space $X$. 
Let $Z$ be a compact metric space. 
Then, the topological embedding $\Lip(Z,U) \to \Lip(Z,X)$ is also an open map.
\end{corollary}
\begin{proof}
Note that $\{f \in C(Z,X) \mid \im\, f \subset U\}$ is open in $C(Z,X)$ with respect to the compact-open topology.
Since the topology on $\Lip(Z,X)$ is finner than the compact-open topology due to $(0)$ of Theorem \ref{thm:top Lip}, a set $\{f \in \Lip(Z,X) \mid \im\, f \subset U\}$ is open in $\Lip(Z,X)$, which is the image of the map $\Lip(Z,U) \to \Lip(Z,X)$.
This completes the proof. 
\end{proof}

\begin{corollary} \label{cor:cpt}
Let $Z$ be a compact metric space and $X$ an arbitrary metric space. 
If a subset $\mathcal K$ of $\Lip(Z,X)$ is compact, then the image set $\im\, \mathcal K = \bigcup_{f \in \mathcal K} \im\, f$ is compact.
\end{corollary}
\begin{proof}
Recall that the evaluation map 
\[
e : Z \times C(Z,X) \ni (z,f) \mapsto f(z) \in X
\]
is continuous in the compact-open topology. 
Let $\mathcal K \subset \Lip(Z,X)$ be a compact set. 
It is also compact in $C(Z,X)$. 
Hence, the set $\im\, \mathcal K = e(Z \times \mathcal K)$ is compact. 
\end{proof}

\begin{corollary} \label{cor:product}
Let $Z$, $X$ and $Y$ be metric spaces. 
Then, the canonical map 
\[
\Lipb(Z,X) \times \Lipb(Z,Y) \to \Lipb(Z,X \times Y)
\] 
is homeomorphic.
\end{corollary}
\begin{proof}
This follows from the properties $(2)$, $(3)$ and $(4)$ of Theorem \ref{thm:top Lip}.
\end{proof}

\begin{proposition} \label{prop:Z cpt} 
Let $Z$ be a compact metric space and $f : X \to Y$ a locally Lipschitz map between metric spaces. 
Then, the map $f_\# : \Lip(Z,X) \to \Lip(Z,Y)$ defined by $g \mapsto f \circ g$ is continuous on each compact set.
\end{proposition}
Note that, in Proposition \ref{prop:Z cpt}, we deal with a {\it locally} Lipschitz map $f : X \to Y$.
Hence, the proposition does not follow from the property (2) of Theorem \ref{thm:top Lip} directly. 
\begin{proof} 
Let $\mathcal K \subset \Lip(Z,X)$ be a compact set.
Let us set $X_0 := \bigcup_{g \in \mathcal K} \im\, g$.
It is compact due to Corollary \ref{cor:cpt}.
By Theorem \ref{thm:top Lip} (2), the inclusion $\iota : X_0 \hookrightarrow X$ induces a topological embedding $\iota_\# : \Lip(Z,X_0) \to \Lip(Z,X)$. 
For each $g \in \mathcal K$, we define a map $g_0 : Z \to X_0$ by $g_0(z) = g(z)$ for $z \in Z$.
Let us set $\mathcal K_0 = \{g_0 \in \Lip(Z,X_0) \mid g \in \mathcal K\}$. 
Then, $\iota_\# : \mathcal K_0 \to \mathcal K$ is bijective. 
Since $\iota_\#$ is a topological embedding, $\mathcal K_0$ is compact. 
Thus, we have the following commutative diagram 
\[
\xymatrix{
&\mathcal K_0 \ar@{^{(}->}[r] \ar[d]^{\iota_\#} &\Lip(Z,X_0) \ar[d]^{\iota_\#} \ar[dr]^{(f|_{X_0})_\#} & \\
&\mathcal K \ar@{^{(}->}[r] &\Lip(Z,X) \ar[r]_{f_\#} &\Lip(Z,Y).
}
\]
Since a locally Lipschitz map restricted to a compact set is Lipschitz, 
$(f |_{X_0})_\#$ is continuous by (2) of Theorem \ref{thm:top Lip}.
Let us take an open set $O$ in $\Lip(Z,Y)$.
Then, the set
\[
\iota_\#^{-1}((f_\# |_{\mathcal K})^{-1}(O)) = \mathcal K_0 \cap (f|_{X_0})_\#^{-1}(O)
\]
is open in $\mathcal K_0$.
Since $\iota_\# |_{\mathcal K_0}$ is a homeomorphism, the map 
\[
f_\# |_{\mathcal K} : \mathcal K \to \Lip(Z,Y)
\]
is continuous.
This completes the proof.
\end{proof}

When a target is a normed abelian group, we obtain

\begin{proposition} \label{prop:TG}
Let $V$ be a normed abelian group and $Z$ a metric space. 
Then, $\Lipb(Z,V)$ is a topological abelian group. 
\end{proposition}
\begin{proof}
Since the addition $+ : V \times V \to V$ is Lipschitz, the induced addition operator 
\[
+_\# : \Lipb(Z, V \times V) \to \Lipb(Z,V)
\]
is continuous by Theorem \ref{thm:top Lip} (2). 
By Corollary \ref{cor:product}, the addition operator
\[
+ : \Lipb(Z,V) \times \Lipb(Z,V) \to \Lipb(Z,V)
\] 
is continuous. 

Since $V \ni v \mapsto - v \in V$ is Lipschitz, the induced map 
\[
\Lipb(Z,V) \ni f \mapsto - f \in \Lipb(Z,V)
\]
is continuous. 
This completes the proof.
\end{proof}


\subsection{Lipschitz measure homology} \label{subsec:Lmh}
By a similar way to define the measure chain complex, the {\it Lipschitz measure chain complex} $\mathscr C_\bullet^\L(X)$ of a metric space $X$ is defined as follows. 
Let  
\[
\mathscr C_k^\L (X) := \mathscr M_\cpt(\Lip(\T^k,X))
\] 
for $k \ge 0$.
The restriction $r_i : \Lip(\T^k,X) \to \Lip(\T^{k-1},X)$ to the $i$-th face is continuous for all $i=0,\dots,k$, due to (3) in Theorem \ref{thm:top Lip}. 
Hence, the boundary $\partial : \mathscr C_k^\L(X) \to \mathscr C_{k-1}^\L(X)$ is defined by the same formula as the usual boundary operator $\partial : \mathscr C_k(X) \to \mathscr C_{k-1}(X)$.
The {\it Lipschitz measure homology} of $X$ is defined by 
\[
\mathscr H_\ast^\L (X) := H_\ast(\mathscr C_\bullet^\L(X)).
\]

Since the inclusion $\Lip(\T^k,X) \hookrightarrow C(\T^k,X)$ is continuous due to $(0)$ of Theorem \ref{thm:top Lip}, any measure $\mu \in \mathscr M_\cpt(\Lip(\T^k,X))$ can be regard as a measure in $\mathscr M_\cpt(C(\T^k,X))$ by push-forward. 
This induces a chain map 
\begin{equation} \label{eq:natural}
\mathscr C_\bullet^\L(X) \to \mathscr C_\bullet(X)
\end{equation}
which is no other than \eqref{eq:incl} in \S \ref{subsec:presice statement}.

Let $Y$ be another metric space. 
For a Lipschitz map $\phi : X \to Y$, the composition $\phi_\# : \Lip(\T^k,X) \to \Lip(\T^k,Y)$ is continuous, due to $(2)$ of Theorem \ref{thm:top Lip}. 
Then, we can define the push-forward 
\[
\phi_\#{}_\# : \mathscr C_k^\L(X) \to \mathscr C_k^\L(Y)
\]
of measures by the continuous map $\phi_\#$. 
The map $\phi_\#{}_\#$ will be written by $\phi_\#$ for shortly. 
It is actually a chain map, due to the definition of the boundary. 

Further, for a {\it locally} Lipschitz map $\phi : X \to Y$, we can define the push-forward $\phi_\# : \mathscr C_\bullet^\L(X) \to \mathscr C_\bullet^\L(Y)$ as follows. 
Let $\mathcal K \subset \Lip(\T^k,X)$ be a compact determination set of $\mu \in \mathscr C_k^\L(X)$. 
The restriction of the composition $\phi_\# : \Lip(\T^k,X) \to \Lip(\T^k,Y)$; $\sigma \mapsto \phi \circ \sigma$ to $\mathcal K$ is continuous, due to Proposition \ref{prop:Z cpt}.
Since the measure $\mu$ is essentially defined on $\mathcal K$, the push-forward of $\mu \lfloor \mathcal K$ under the continuous map $\phi_\# |_{\mathcal K}$ is defined. 
We denote it by $\phi_\# \mu$. 
This construction is actually well-defined due to the following

\begin{lemma} \label{lem:conti on cpt} 
Let $\phi : X \to Y$ be a locally Lipschitz map between metric spaces. 
Then, the above construction of the push-forward $\mathscr C_k^\L(X) \ni \mu \mapsto \phi_\# \mu \in \mathscr C_k^\L(Y)$ does not depend on the choice of a compact determination set $\mathcal K$ of $\mu$. 
Further, $\phi_\# : \mathscr C_\bullet^\L(X) \to \mathscr C_\bullet^\L(Y)$ is a chain map.
\end{lemma}
\begin{proof}
Let $\mathcal K_1$ and $\mathcal K_2$ be two compact determination sets of $\mu$ in $\Lip(\T^k,X)$.
Then, $\mathcal K_1 \cup \mathcal K_2$ is also a compact determination set. 
Hence, we may assume that $\mathcal K_1 \subset \mathcal K_2$.
Let us consider the push-forwards $\nu_i := (\phi_\# |_{\mathcal K_i})_\# (\mu \lfloor \mathcal K_i)$ of $\mu \lfloor \mathcal K_i \in \mathscr M_\cpt(\mathcal K_i)$ by the continuous maps $\phi_\# |_{\mathcal K_i} : \mathcal K_i \to \Lip(\T^k,Y)$ for $i=1,2$. 
Then, $\nu_i \in \mathscr M_\cpt(\phi_\#(\mathcal K_i))$.
Further, we regard them as signed Borel measures on $\Lip(\T^k,Y)$.
Let $\mathcal A$ be a Borel set in $\Lip(\T^k,Y)$.
Then, we obtain 
\[
\nu_2(\mathcal A) = \mu((\phi_\# |_{\mathcal K_2})^{-1} (\mathcal A)) 
= \mu (\phi_\#^{-1}(\mathcal A) \cap \mathcal K_2).
\]
Since $\mathcal K_1$ is a determination set of $\mu$, we have 
\begin{align*}
\mu (\phi_\#^{-1}(\mathcal A) \cap \mathcal K_2) 
&= \mu (\phi_\#^{-1}(\mathcal A) \cap \mathcal K_2 \cap \mathcal K_1) + \mu (\phi_\#^{-1}(\mathcal A) \cap \mathcal K_2 - \mathcal K_1) \\
&= \mu (\phi_\#^{-1}(\mathcal A) \cap \mathcal K_1) \\
&= \nu_1(\mathcal A).
\end{align*}
Therefore, $\nu_1 = \nu_2$ as the signed Borel measures on $\Lip(\T^k, Y)$, which are denoted by $\phi_\# \mu$.
By the construction, $\phi_\# \mu$ is of finite total variation and has a compact determination set.

It is easily verified that the map $\phi_\# : \mathscr C_\bullet^\L(X) \to \mathscr C_\bullet^\L(Y)$ is a chain map. 
This completes the proof.
\end{proof}

The map \eqref{eq:natural} is natural, i.e., 
$\phi_\# \circ (\mathscr C_\bullet^\L (X) \to \mathscr C_\bullet(X)) = (\mathscr C_\bullet^\L (Y) \to \mathscr C_\bullet(Y)) \circ \phi_\#$ holds, for every locally Lipschitz map $\phi : X \to Y$.
Hence, it induces a natural transformation 
\begin{equation*} 
\mathscr H_\ast^\L \to \mathscr H_\ast
\end{equation*}
on the category of all metric spaces and all locally Lipschitz maps.

For a pair $(X,A)$ of metric spaces, 
the inclusion $A \hookrightarrow X$ induces a topological embedding $\Lip(\T^k,A) \to \Lip(\T^k,X)$ for every $k \ge 0$, due to $(2)$ of Theorem \ref{thm:top Lip}. 
So, it induces an injective chain map 
\[
\mathscr C_\bullet^\L (A) \to \mathscr C_\bullet^\L(X).
\]
Thus, $\mathscr C_\bullet^\L(A)$ is regarded as a subcomplex of $\mathscr C_\bullet^\L(X)$. 
We set 
\[
\mathscr C_\bullet^\L(X,A) = \mathscr C_\bullet^\L(X) / \mathscr C_\bullet^\L(A)
\]
and call it the Lipschitz measure complex of $(X,A)$. 
Its homology is denoted by
\[
\mathscr H_\ast^\L (X,A)
\]
called the Lipschitz measure homology of $(X,A)$. 
There are also natural maps 
\[
\mathscr C_\bullet^\L (X,A) \to \mathscr C_\bullet (X,A)
\]
and 
\[
\mathscr H_\ast^\L (X,A) \to \mathscr H_\ast (X,A)
\]
on the category of all pairs of metric spaces and all locally Lipschitz maps.

\begin{remark} \upshape
By $(5)$ of Theorem \ref{thm:top Lip}, there are canonical identifications $C(\T^0,X) = \Lip(\T^0,X) = X$ for any metric space $X$. 
Hence, we can identify $\mathscr C_0(X)$ and $\mathscr C_0^\L(X)$ with $\mathscr M_\cpt(X)$.
\end{remark}

The following is independent on main results.

\begin{proposition} \label{prop:mL}
Let $X$ be a metric space which has no nonconstant Lipschitz curves. 
Namely, if $\sigma : [0,1] \to X$ is Lipschitz, then $\sigma$ is a constant map. 
Then, $\mathscr H_k^\L(X) = 0$ for all $k \ge 1$.
\end{proposition}
\begin{proof}
Let $X$ be assumed as in the assumption. 
Then, $\Lip(\T^k,X) = \Lip(\T^0,X) = X$ for all $k \ge 1$. 
Hence, $\mathscr C_k^\L(X) = \mathscr C_0^\L(X) = \mathscr M_\cpt(X)$ for $k \ge 1$. 
The boundary map $\partial = (\partial_k)_{k \ge 1}$ of the complex $\mathscr C_\bullet^\L(X)$ becomes  
\[
\partial_k = \left\{ 
\begin{aligned}
&0 && \text{ if } k \text{ is odd}, \\
&\mathrm{id_{\mathscr M_\cpt(X)}} && \text{ if } k \text{ is even}.
\end{aligned}
\right.
\]
This implies $\mathscr H_k^\L(X) = 0$ for any $k \ge 1$.
\end{proof}
Due to Proposition \ref{prop:mL}, the Lipschitz measure homology is known to be actually dpendeing on the metric structures. 
For instance, if $X = (X,d)$ is a Riemannian manifold with the distance function $d$ induced from the Riemannian metric, then by Theorem \ref{thm:main thm} and Corollary \ref{cor:sh}, $\mathscr H_\ast^\Lip(X)$ is isomorphic to the singular real homology. 
On the other hands, one can easily to show that a snowflake version $X^\alpha = (X, d^\alpha)$ for $0 < \alpha < 1$ has no nonconstant Lipschitz curve (see \cite{He}). 
Hence, by Proposition \ref{prop:mL}, we have $\mathscr H_k^\Lip(X^\alpha) = 0$ for $k \ge 1$.

\subsection{A natural map from $\mathscr C_\bullet^\L$ to $\mathbf N_\bullet^\cpt$} 
\label{subsec:natural map}
In this subsection, we construct a natural chain map from the Lipschitz measure chain complex to the current chain complex. 

Let us denote by $X$ a metric space. 
For $\sigma \in \Lip(\T^k,X)$, Riedweg and Sch\"appi (\cite{RS}) considered a functional $[\sigma] : \mathcal D^k(X) \to \mathbb R$ defined by 
\[
[\sigma] (f d \pi) = \int_{\T^k} f \circ \sigma (s) \det (\nabla (\pi \circ \sigma(s))) \, d \mathcal L^k(s)
\]
for each $f d \pi \in \mathcal D^k(X)$, 
where $\mathcal L^k$ is the $k$-dimensional Lebesgue measure. 
Here, the gradient $\nabla (\pi \circ \sigma)$ is defined for almost all points on $\T^k$, due to Rademacher's theorem. 
It is also represented as $[\sigma] = \sigma_\# \jump{1_{\T^k}}$, 
where $\jump{1_{\T^k}}$ is a $k$-current in $\T^k$ given by 
\[
\jump{1_{\T^k}}(g d \tau) = \int_{\T^k} g \det(\nabla \tau) \,d \mathcal L^k
\] 
for $g d \tau \in \mathcal D^k(\T^k)$ (see \cite{AK}).
By the definition, $[\sigma]$ is a normal $k$-current having compact support contained in the image of $\sigma$.

\begin{lemma} \label{lem:conti}
For each $f d \pi \in \mathcal D^k(X)$, the functional 
\begin{equation*} 
\Lip(\T^k,X) \ni \sigma \mapsto [\sigma](f d \pi) \in \mathbb R
\end{equation*}
is continuous in the topology on $\Lip(\T^k,X)$.
\end{lemma}
\begin{proof}
Let us fix $f d \pi \in \mathcal D^k(X)$.
By Theorem \ref{thm:top Lip} $(0)$, $\Lip(\T^k,X)$ is metrizable. 
Hence, it suffices to show that $[\,\cdot\,](f d\pi)$ is sequentially continuous.
Let a sequence $\sigma_j$ converge to $\sigma$ in $\Lip(\T^k,X)$.
By Theorem \ref{thm:top Lip} $(1)$, $f \circ \sigma_j \to f \circ \sigma$ and $\pi \circ \sigma_j \to \pi \circ \sigma$ uniformly as $j \to \infty$, and $\sup_j \Lip(\pi \circ \sigma_j) < \infty$.
Hence, we have 
\[
[\sigma_j](f d\pi) = \jump{1_{\T^k}} (f \circ \sigma_j \, d \pi \circ \sigma_j) \to \jump{1_{\T^k}}(f \circ \sigma\, d\pi \circ \sigma) = [\sigma](f d \pi), 
\]
as $j \to \infty$, 
because $\jump{1_{\T^k}}$ is a current.
This completes the proof.
\end{proof}

\begin{remark} \upshape \label{rem:not conti}
When $\Lip(\T^k,X)$ is merely equipped with the compact-open topology, the functional $[\,\cdot\,](f d\pi)$ is not continuous in general. 

Let us consider a family $\{u_\epsilon\}_{\epsilon > 0}$ of smooth maps  from $\T^k$ to $\mathbb R^k$, for $k \ge 2$ such that 
$u_\epsilon \to u$ as $\epsilon \to 0$ uniformly, and $\det (\nabla u_\epsilon) \rightharpoonup 1$ in $L^\infty(\T^k)$ weakly$^\ast$, as $\epsilon \to 0$, where $u(x) = 0 \in \mathbb R^k$ for $x \in \T^k$.
Here, $\T^k$ is considered as a convex subset of $\mathbb R^k$.
Such a family can be found in \cite{ABF}, which is actually given by
\[
\left\{
\begin{aligned}
(u_\epsilon)_1(x) &= \sqrt{2 \epsilon} \sin (x_1 / \epsilon), \\
(u_\epsilon)_2(x) &= x_2 \sqrt{2 \epsilon} \cos(x_1 / \epsilon), \\
(u_\epsilon)_i(x) &= x_i \hspace{10pt} (3 \le i \le k). 
\end{aligned}
\right.
\]
Obviously, $\det(\nabla u) = 0$.
Note that each $u_\epsilon$ is Lipschitz and 
\[
\sup_{\epsilon > 0} \Lip(u_\epsilon) = \infty.
\]

By using this family, we show the discontinuity of $[\,\cdot\,](f d\pi)$ on $\Lip(\T^k,\mathbb R^k)$ in the compact-open topology, for a particular choice of $f d \pi \in \mathcal D^k(\mathbb R^k)$.
Indeed, we choose $f(x) = 1$ and $\pi(x) = x$ for $x \in \mathbb R^k$.
Then, we have 
\[
[u_\epsilon](f d \pi) = \int_{\T^k} \det(\nabla u_\epsilon) \, d \mathcal L^k \to \mathcal L^k(\T^k) \hspace{10pt} \text{ as } \epsilon \to 0. 
\]
Hence, $[u_\epsilon](f d\pi)$ does not converge to $[u](f d\pi) = 0$ in the compact-open topology. 
\end{remark}

\begin{definition} \upshape
For $\mu \in \mathscr C_k^\L(X)$, we define a functional $T^\mu : \mathcal D^k(X) \to \mathbb R$ by 
\begin{equation} \label{eq:Tmu}
T^\mu(f d \pi) = \int_{\Lip(\T^k,X)} [\sigma](f d \pi) \, d \mu
\end{equation}
for $f d \pi \in \mathcal D^k(X)$.
\end{definition}
Due to Lemma \ref{lem:conti}, the integral \eqref{eq:Tmu} is well-defined.

\begin{remark} \upshape \label{rem:not integrable}
Let us denote by $\mathfrak T_\CO$ the compact-open topology on $\Lip(\T^k,X)$. 
We show that the integral \eqref{eq:Tmu} is not well-defined for 
a finite Borel measure on $(\Lip(\T^k,X), \mathfrak T_\CO)$ with compact support.

We consider smooth functions $v_\epsilon, v_0 \in C^1(\T^2, \mathbb R^2)  \subset \Lip(\T^2,\mathbb R^2)$ for $\epsilon > 0$, similar to functions appeared in Remark \ref{rem:not conti}, defined by 
\[
\left\{
\begin{aligned}
(v_\epsilon)_1 (x) &= \sqrt{\epsilon} \sin(x_1 / \epsilon^2) \\
(v_\epsilon)_2 (x) &= \sqrt{\epsilon} x_2 \cos(x_1 / \epsilon^2),
\end{aligned}
\right.
\] 
and $v_0 \equiv 0$.
Here, we identify $\T^2$ with $\{(x_1,x_2) \in \mathbb R^2 \mid x_1 \in [-1,1], x_2 \in [0, \sqrt 3 x_1]\}$.
Then, the map 
\[
V : [0,1] \ni \epsilon \mapsto v_\epsilon \in \Lip(\T^2,\mathbb R^2)
\]
is continuous in $\mathfrak T_\CO$. 
Hence, its image $K := \{v_\epsilon \mid 0 \le \epsilon \le 1\}$ is compact in $\mathfrak T_\CO$.  
We consider the push-forward $\mu := V_\# \mathcal L^1$ of the Lebesgue measure $\mathcal L^1$ on $[0,1]$ with respect to the map $V$. 
By the definition, $\mu$ is a finite positive Borel measure on $(\Lip(\T^2,\mathbb R^2), \mathfrak T_\CO)$ 
supported in the compact set $K$. 
For $f d\pi \in \mathcal D^2(\mathbb R^2)$ with $f = 1$ and $\pi = \mathrm{id}$, we verify that $[\,\cdot\,](f d\pi)$ can not integrable in $\mu$. 
Suppose that 
\[
\Lip(\T^2, \mathbb R^2) \ni \sigma \mapsto [\sigma](f d \pi) \in \mathbb R
\]
is Borel measurable in $\mathfrak T_\CO$. 
Then, we have 
\begin{align*}
\int_{\Lip(\T^2, \mathbb R^2)} [\sigma](f d \pi) \, d \mu(\sigma) 
&= \int_0^1 [v_\epsilon] (f d\pi) \, d \epsilon \\
&= \int_0^1 \int_{\T^2} \det(\nabla v_\epsilon)(x) \,d x d \epsilon \\
&= C \int_0^1 \frac{d \epsilon}{\epsilon} +  C',
\end{align*}
where $C$ and $C'$ are constants with $C > 0$. 
Therefore, the functional of $[\,\cdot\,](f d \pi)$ is not integrable in $\mu$.
\end{remark}

\begin{theorem} \label{thm:Tmu}
Let $X$ be an arbitrary metric space. 
For $\mu \in \mathscr C_k^\L(X)$, $T^\mu$ is a normal $k$-current in $X$ of compact support.
Further, $\mathscr C_\bullet^\L(X) \ni \mu \mapsto T^\mu \in \N_\bullet^\cpt(X)$ is a chain map.
\end{theorem}
\begin{proof}
Let us take $\mu \in \mathscr C_k^\L(X)$. 
Let $\mathcal K$ denote its compact determination set in $\Lip(\T^k,X)$. 
By Corollary \ref{cor:cpt}, the image set $\im\, \mathcal K$ is compact.
We first prove that $T^\mu$ is a $k$-current in $X$.
By the definition, $T^\mu$ is multilinear on $\mathcal D^k(X)$. 
Since each $[\sigma]$ satisfies the locality, $T^\mu$ satisfies the locality.
To verify that $T^\mu$ has finite mass, we estimate the absolute value $|T^\mu(f d \pi)|$ as follows.
\begin{align*}
|T^\mu (f d \pi)| 
&\le \int_{\mathcal K} \int_{\T^k} \left|f \circ \sigma \det (\nabla(\pi \circ \sigma)) \right| \, d \mathcal L^k d |\mu| \\
&\le \Lip(\sigma)^k \prod_{i=1}^k \Lip(\pi_i) \int_{\mathcal K} \int_{\T^k} |f \circ \sigma| \, d \mathcal L^k d |\mu| \\
&\le L^k \prod_{i=1}^k \Lip(\pi_i) \int_{\mathcal K} \int_{\T^k} |f \circ \sigma| \, d \mathcal L^k d |\mu|.
\end{align*}
Here, $L := \sup_{\sigma \in \mathcal K} \Lip(\sigma)$, which is finite by (1) of Theorem \ref{thm:top Lip}.
Now, we note that the map 
\[
\mathcal K \ni \sigma \mapsto \int_{\T^k} g \circ \sigma \, d \mathcal L^k
\]
is continuous for each $g \in C(\im\, \mathcal K)$, where $C(\im\,\mathcal K)$ is the space of all continuous functions from $\im\, \mathcal K$ to $\mathbb R$.
Hence, a functional 
\begin{equation*} 
C(\im\, \mathcal K) \ni g \mapsto \int_{\mathcal K} \int_{\T^k} g \circ \sigma \,d \mathcal L^k d |\mu| \in \mathbb R
\end{equation*}
is well-defined. 
Since this functional is positive linear, 
due to the Riesz-Markov-Kakutani representation theorem, 
there is a Borel
measure $\nu$ 
on $X$ supported in the compact set $\im\, \mathcal K$ such that 
\[
|T^\mu(f d \pi)| \le L^k \prod_{i=1}^k \Lip(\pi_i) \int_X |f| \, d \nu
\]
holds, for every $f d \pi \in \mathcal D^k(X)$. 
So, $T^\mu$ is proved to have finite mass.

Next, we prove that $T^\mu$ is continuous in the sense of current. 
Let us take sequences $\pi_i^j \in \Lip(X)$ converging to $\pi_i$ as $j \to \infty$ for $i = 1, \dots, k$ with $\sup_{i,j}\Lip(\pi_i^j) < \infty$.
We set $L' := \sup_{i,j}\Lip(\pi_i^j)$.
Then, $[\,\cdot\,](f d\pi^j)$ converges to $[\,\cdot\,](f d\pi)$ pointwise on $\Lip(\T^k,X)$ as $j \to \infty$.
Further, we have 
\begin{align*}
\left| \int_{\mathcal K} [\sigma](f d \pi^j) \, d |\mu| \right| 
&\le \int_{\mathcal K} \left|[\sigma](f d\pi^j)\right| d|\mu| \\
&\le \int_{\mathcal K} \int_{\T^k} \left| f \circ \sigma \det(\nabla (\pi^j \circ \sigma))\right| d \mathcal L^k\, d|\mu| \\
&\le (L L')^k \int_X |f| \,d \nu
\end{align*}
for all $j \in \mathbb N$.
Therefore, by the dominated convergence theorem, 
\[
\lim_{j \to \infty} T^\mu(f d\pi^j) = T^\mu(f d\pi)
\]
holds. 
Thus, the continuity of $T^\mu$ is proved.

To prove that $T^\mu$ is normal, let us compare $\partial T^\mu$ with $T^{\partial \mu}$.
For any $f d \pi \in \mathcal D^{k-1}(X)$, we have 
\begin{align*}
\partial T^\mu(f d \pi) &= T^\mu( d f \wedge d \pi) = \int_{\Lip(\T^k,X)} [\sigma](d f \wedge d \pi) \, d \mu \\
&= \int_{\Lip(\T^k,X)} [\partial \sigma](f d \pi) \, d \mu = \sum_{i=0}^k (-1)^i \int_{\Lip(\T^k,X)} [r_i \sigma] (f d \pi) \, d \mu,
\end{align*}
where $r_i$ denotes the restriction to the $i$-th face. 
On the other hands, we obtain 
\begin{align*}
T^{\partial \mu}(f d \pi) &= \int_{\Lip(\T^{k-1},X)} [\tau] (f d \pi) \, d (\partial \mu) \\
&= \sum_{i=0}^k (-1)^i \int_{\Lip(\T^{k-1},X)} [\tau] (f d \pi) \, d \left( r_i {}_\# \mu \right) \\
&= \sum_{i=0}^k (-1)^i \int_{\Lip(\T^k,X)} [r_i \sigma] (f d \pi) \, d \mu.
\end{align*}
Therefore, $\partial T^\mu = T^{\partial \mu}$, and hence, $T^\mu$ is normal.
Further, we know that $\mu \mapsto T^\mu$ is a chain map. 
This completes the proof. 
\end{proof}

The following is fundamental and important.
\begin{lemma} \label{lem:0-th}
Let $X$ be an arbitrary metric space.
The map $\mathscr C_0^\L(X) \ni \mu \mapsto T^\mu \in \N_0^\cpt(X)$ is isomorphic.
\end{lemma}
\begin{proof}
We prove that the map $T^{(\,\cdot\,)} : \mu \mapsto T^\mu$ is surjective.
Let us take $T \in \N_0^\cpt(X)$. 
Let $K$ denote its compact support. 
Then, $T \lfloor K$ and $T$ are same as currents. 
Here, $T \lfloor K$ is the linear functional defined by 
\[
T \lfloor K (f) = T(1_K f)
\]
for all bounded Borel functions $f : X \to \mathbb R$, where $1_K$ is the characteristic function of $K$.
Further, $T \lfloor K$ is also considered as a normal current in $K$. 
Due to the finite mass axiom, $T = T \lfloor K$ is continuous on the space $C(K)$ with respect to the uniform topology. 
By the Riesz-Markov-Kakutani theorem, there is a unique regular Borel measure $\mu$ on $K$ such that 
\[
T \lfloor K(f) = \int_K f \,d\mu
\]
holds for all $f \in C(K)$, and that the total variation of $\mu$ equals the operator norm of $T$ as the linear functional $T : C(K) \to \mathbb R$. 

Let us regard $\mu$ as a signed Borel measure $\bar \mu$ on $X$. 
Namely, it is defined by 
\[
\bar \mu (A) = \mu(A \cap K)
\]
for every Borel set $A$ of $X$.
Obviously, $\bar \mu$ has the compact determination set $K$ and is of finite total variation.
Namely, $\bar \mu \in \mathscr M_\cpt(X) = \mathscr C_0(X) = \mathscr C_0^\L(X)$.
Then, for every $f \in \Lipb(X)$, we have 
\[
T^{\bar \mu} (f) = \int_X f \,d\bar \mu = \int_K f \,d \mu = T \lfloor K (f |_K) = T(f).
\]
Therefore, $T^{(\cdot)}$ is surjective.
If $T^\mu = 0$, then obviously $\mu = 0$.
Hence, the map $T^{(\cdot)}$ is injective.
This completes the proof. 
\end{proof}


\section{Proof of Theorem \ref{thm:main thm}} \label{sec:proof:main thm}

In this section, we prove Theorem \ref{thm:main thm} by using cosheaf theory. 

\subsection{Review of cosheaf theory} \label{subsec:cosheaf}
Let us recall the notion of cosheaf. 
Let $X$ be a topological space and $\mathsf O(X)$ the set of all open sets in $X$. 
We regard $\mathsf O(X)$ as a category, by assigning an arrow $U \to V$ if and only if $U \subset V$.
Let $\mathsf {Ch}$ denote the category of chain complices of real vector spaces and chain maps.
Here, we consider chain complices indexed by $\mathbb Z$.
For a chian complex $(C_k)_{k \ge 0}$ (resp. $(D_k)_{k \ge -1}$) of nonnegative degrees (resp. of degrees not less than $-1$), we extend it to a chian complex of integer degrees by setting $C_\ell = 0$ for all $\ell < 0$ (resp. $D_\ell = 0$ for all $\ell <-1$).
A covariant functor $\mathfrak A : \mathsf O(X) \to \mathsf {Ch}$ is called a {\it precosheaf} (of $\mathsf{Ch}$-valued).
We denote the map $\mathfrak A(U \to V) : \mathfrak A(U) \to \mathfrak A(V)$ by $i_{V,U}$. 
The precosheaf $\A$ is called a {\it cosheaf} if, 
for any family $\{U_\alpha\}$ of open sets in $X$, 
\[
\bigoplus_{\alpha, \beta} \A(U_\alpha \cap U_\beta) \xrightarrow{\Phi_1} \bigoplus \A(U_\alpha) \xrightarrow{\Phi_0} \A(U) \to 0
\]
is exact, where $U = \bigcup_\alpha U_\alpha$, $\Phi_0 = \sum_\alpha i_{U,U_\alpha}$ and $\Phi_1 = \sum_{\alpha,\beta} i_{U_\beta, U_\alpha \cap U_\beta} - i_{U_\alpha, U_\alpha \cap U_\beta}$. 

Let $\mathsf{Vect}$ denote the category of all vector spaces and all linear maps over $\mathbb R$.
We define two functors $H_m, \tilde H_m : \mathsf{Ch} \to \mathsf{Vect}$ for each $m \ge 0$ as follows. 
Let $C = (C_k)_{k \in \mathbb Z}$ be a chain complex. 
Then, $H_m(C)$ denotes the $m$-th homology of the restricted chain complex $(C_k)_{k \ge 0}$ and $\tilde H_m(C)$ denotes the usual $m$-th homology of $C$: 
\begin{align*}
&H_m(C) := \left\{ 
\begin{aligned}
&\!\ker \partial / \im\, \partial &&\text{ if } m \ge 1\\
&C_0 / \im\, \partial &&\text{ if } m = 0, 
\end{aligned}
\right.
\\
&
\tilde H_m(C) := 
\ker \partial / \im\, \partial \hspace{22pt}\text{ for any } m.
\end{align*}
For a precosheaf $\A : \mathsf O(X) \to \mathsf {Ch}$, we have precosheaves $H_m(\A), \tilde H_m(\A) : \mathsf O(X) \to \mathsf{Vect}$ for all $m \ge 0$. 

A precosheaf $\A''$ on $X$ is said to be {\it flabby} if $\A''(O) \to \A''(X)$ is injective for every $O \in \mathsf O(X)$.
The precosheaf $\A''$ is said to be {\it locally trivial} if for any $x \in X$ and $O \in \mathsf O(X)$ with $x \in O$, there is $O' \in \mathsf O(X)$ with $x \in O' \subset O$ such that the map $i_{O,O'} : \A''(O') \to \A''(O)$ is zero.

In \cite{Mi}, the following is proved. 
\begin{theorem}[\cite{Mi}] \label{thm:coincidence}
Let $X$ be a paracompact topological space. 
Let $\A$ and $\A'$ be precosheaves on $X$ of $\mathsf {Ch}$-valued.  
Suppose the following conditions: 
\begin{itemize}
\item[(i)] the precosheaves $\A_m$ and $\A_m'$ are flabby cosheaves of $\mathsf{Vect}$-valued for $m \ge 0$,
\item[(ii)] there exists a natural transformation $\phi : \A \to \A'$ such that $\A_0(O) \to \A_0'(O)$ is surjective and $\A_{-1}(O) \to \A_{-1}'(O)$ is injective for each $O \in \mathsf O(X)$,
\item[(iii)] the precosheaves $\tilde H_m(\A)$ and $H_n(\A')$ are locally trivial on $X$ for all $m \ge 0$ and $n \ge 1$. 
\end{itemize}
Then, $\phi$ induces an isomorphism 
\[
\phi_\ast : H_m(\A) \to H_m(\A')
\]
between precosheaves for each $m \ge 0$.
\end{theorem}


The following is convenience.
\begin{proposition}[{\cite[Chap.\!\! VI]{B}}] \label{prop:cosheaf}
Let $\A$ be a precosheaf on a topological space $X$. 
Then, $\A$ is a cosheaf if and only if it satisfies the following two conditions: 
\begin{itemize}
\item For any two open sets $U$ and $V$ in $X$, the short sequence 
\[
\A(U \cap V) \xrightarrow{\Phi_1} \A(U) \oplus \A(V) \xrightarrow{\Phi_0} \A(U \cup V) \to 0
\]
is exact, where $\Phi_1 = -i_{U, U \cap V} + i_{V, U \cap V}$ and $\Phi_0 = i_{U \cup V, U} + i_{U \cup V, V}$.
\item 
Let $\{U_\alpha\}$ be a family of open sets of $X$, directed upwards by inclusion.
Namely, for any indices $\alpha, \alpha'$, there exists $\alpha''$ such that $U_\alpha \cup U_{\alpha'} \subset U_{\alpha''}$. 
We set $U = \bigcup_\alpha U_\alpha$.
Then, the maps $i_{U,U_\alpha} : \A(U_\alpha) \to \A(U)$ induces an isomorphism 
\[
\varinjlim i_{U,U_\alpha} : \varinjlim \A(U_\alpha) \to \A(U).
\]
\end{itemize}
\end{proposition}

\subsection{$\mathscr C_\bullet$ and $\mathscr C_\bullet^\L$ as cosheaves} \label{subsec:proof:cosheaf}
In this section, we prove that the functors $\mathscr C_\bullet$ and $\mathscr C_\bullet^\L$ are flabby cosheaves on a metric space. 

\begin{lemma} \label{lem:U'}
Let $S$ be a metrizable space. Then, the following holds. 
\begin{itemize}
\item[(1)]
For a compact set $K$ and an open set $U$ in $S$ with $K \subset U$, 
there is an open set $U'$ in $S$ such that $K \subset U' \subset \overline{U'} \subset U$.
Here, $\overline{U'}$ is the closure of $U'$ in the topology of $S$. 
\item[(2)]
For a compact set $K$ and open sets $U$ and $V$ in $S$ with $K \subset U \cup V$, 
there exists an open set $W$ such that $\overline W \subset V$ and $K - W \subset U$.
\end{itemize}
\end{lemma}
\begin{proof}
We fix a metric on $S$ which is compatible to the topology.
We prove $(1)$.
Let $U$ be an open set in $S$ and $K$ a compact set with $K \subset U$.
Let us consider the distance function $\rho$ from $X-U$.
Then, $\{\rho > 0\} = U$.
Since $K$ is compact, $\epsilon := \min_K \rho > 0$.
Let us set $U' := \{\rho > \epsilon / 2\}$.
Then, it satisfies the desired property of (1).

We prove $(2)$. 
Let $U$ and $V$ be open sets in $S$ and $K$ a compact set with $K \subset U \cup V$. 
Let us consider the distance function $\rho$ from $S - V$. 
Then, there exists $\epsilon > 0$ such that $K \cap \{\rho \le \epsilon\} \subset U$. 
Indeed, if this claim is false, then there exists a sequence $x_j \in K \cap \{\rho \le 1/ j\}$ such that $x_j \not\in U$.
Since $K$ is compact, there is a subsequence of $x_j$ converges to some point $x \in K \cap \{\rho = 0\}$. 
Then, the point $x$ satisfies $x \in K$, $x \not\in U$ and $x \not\in V$, which is a contradiction. 
We set $W = \{\rho > \epsilon\}$. 
Then, it satisfies the desired property of (2).
\end{proof}

\begin{lemma} \label{lem:UV}
Let $S$ be a metrizable space. 
Then, for any $U, V \in \mathsf O(S)$, 
\[
\mathscr M_\cpt(U \cap V) \xrightarrow{\Phi_1} \mathscr M_\cpt(U) \oplus \mathscr M_\cpt(V) \xrightarrow{\Phi_0} \mathscr M_\cpt(U \cup V) \to 0
\]
is exact, where $\Phi_1(\xi) = (\xi, -\xi)$ and $\Phi_0(\mu,\nu) = \mu + \nu$.
\end{lemma}
\begin{proof}
Let $\xi \in \mathscr M_\cpt(U \cup V)$ with a compact determination set $K \subset U \cup V$.
By Lemma \ref{lem:U'}, there exists an open set $W$ in $S$ such that 
\[
\overline W \subset V \text{ and } K - W \subset U.
\]
The restriction $\mu := \xi \lfloor (K - W)$ can be regarded as a singed measure on $U$ with the compact determination set $K - W$.
Let us consider the restriction $\nu := \xi \lfloor (K \cap W)$ which is regarded as a signed measure on $V$. 
It has a determination set $K \cap W$.  
Hence, the compact set $K \cap \overline W$ is also a determination set of $\nu$.
Therefore, we have $\mu \in \mathscr M_\cpt(U)$ and $\nu \in \mathscr M_\cpt(V)$ with $\mu + \nu = \xi$.
Hence, $\Phi_0$ is surjective.

Let us take $\mu \in \mathscr M_\cpt(U)$ and $\nu \in \mathscr M_\cpt(V)$ with $\mu + \nu = 0 \in \mathscr M_\cpt(U \cup V)$.
Let $K$ and $L$ be compact sets in $U$ and $V$ which are determination sets of $\mu$ and $\nu$, respectively.
By Lemma \ref{lem:U'}, there is an open subset $V'$ in $S$ such that 
\[
L \subset V' \subset \overline{V'} \subset V.
\]
Then, $\mu \lfloor (K - V') = 0$ in $\mathscr M_\cpt(U)$. 
We set $\xi := \mu \lfloor \overline{V'}$ which is a signed measure on the compact set $K \cap \overline{V'}$. 
Thus, we can regard $\xi$ as an element of $\mathscr M_\cpt(U \cap V)$. 
By the construction, we have $\Phi_1(\xi) = (\mu, \nu)$.
\end{proof}

\begin{lemma} \label{lem:limit}
Let $\{U_\alpha\}$ be a family of open sets in a topological space $S$ which is directed upwards by inclusions. 
Then, the map $\varinjlim \mathscr M_\cpt(U_\alpha) \to \mathscr M_\cpt(U)$ induced by the maps $\mathscr M_\cpt(U_\alpha) \to \mathscr M_\cpt(U)$, is isomorphic, where $U = \bigcup_\alpha U_\alpha$.
\end{lemma}
\begin{proof}
Since all maps $\mathscr M_\cpt(U_\alpha) \to \mathscr M_\cpt(U)$ are injective, the map $\varinjlim \mathscr M_\cpt(U_\alpha) \to \mathscr M_\cpt(U)$ is injective. 
Let us take $\mu \in \mathscr M_\cpt(U)$ with a compact determination set $K \subset U$. 
Since $\{U_\alpha\}$ is directed upwards by inclusions, there is $\alpha$ such that $K \subset U_\alpha$.
Then, we can regard $\mu$ as a measure in $\mathscr M_\cpt(U_\alpha)$.
Hence, the map $\varinjlim \mathscr M_\cpt(U_\alpha) \to \mathscr M_\cpt(U)$ is surjective. 
\end{proof}

\begin{corollary} \label{cor:mc is cosheaf}
For a topological space $X$, the correspondence $\mathsf O(X) \ni O \mapsto \mathscr C_\bullet(O) \in \mathsf{Ch}$ is a flabby cosheaf. 
For a metric space $X$, the correspondence $\mathsf O(X) \ni O \mapsto \mathscr C_\bullet^\L(O) \in \mathsf{Ch}$ is a flabby cosheaf. 
\end{corollary}
\begin{proof}
This follows from Theorem \ref{thm:top Lip}, Corollary \ref{cor:open emb}, 
Lemmas \ref{lem:UV} and \ref{lem:limit} and Proposition \ref{prop:cosheaf}.
\end{proof}

\subsection{Reduced homologies} \label{subsec:reduced}
Let us consider an augmentation of the measure chain complex of a metric space $X$ defined by 
\[
\tilde \partial_0: \mathscr C_0(X) = \mathscr M_\cpt(X) \ni \mu \mapsto \mu(X) \in \mathbb R.
\]
For a metric space $Y$, an augmentation of the Lipschitz measure chain complex is also defined by $\tilde \partial_0 : \mathscr C_0^\L(Y) \ni \mu \mapsto \mu(Y) \in \mathbb R$.
\begin{lemma} \label{lem:augmentation}
The maps $\tilde \partial_0$ are actually augmentations of $\mathscr C_\bullet(X)$ and $\mathscr C_\bullet^\L(Y)$ for a topological space $X$ and a metric space $Y$.
\end{lemma}
\begin{proof}
Let $\mu \in \mathscr C_1(X)$. 
Then, 
\[
\tilde \partial_0 \partial_1 \mu = \partial_1 \mu(X) = \int_X 1 \,d (\partial_1 \mu) = \int_{C(I,X)} 1 - 1 \, d\mu = 0.
\]
Similarly, we have $\tilde \partial_0 \partial_1 \mu = 0$ for $\mu \in \mathscr C_1^\L(Y)$.
\end{proof}

The augmented (Lipschitz) measure chain complices by $\tilde \partial_0$ are denoted by $\tilde {\mathscr C}_\bullet$ and $\tilde {\mathscr C}_\bullet^\L$.
Namely, 
\begin{align*}
\tilde {\mathscr C}_\bullet 
&= (\cdots \xrightarrow{\partial_{k+1}} \mathscr C_k \xrightarrow{\partial_k} \mathscr C_{k-1} \xrightarrow{\partial_{k-1}} \cdots \xrightarrow{\partial_1} \mathscr C_0 \xrightarrow{\tilde \partial_0} \mathbb R) \\
\tilde {\mathscr C}_\bullet^\L 
&= (\cdots \xrightarrow{\partial_{k+1}} \mathscr C_k^\L \xrightarrow{\partial_k} \mathscr C_{k-1}^\L \xrightarrow{\partial_{k-1}} \cdots \xrightarrow{\partial_1} \mathscr C_0^\L \xrightarrow{\tilde \partial_0} \mathbb R).
\end{align*}
Their homologies are written by 
\[
\tilde{\mathscr H}_\ast = \tilde H_\ast (\tilde{\mathscr C}_\bullet) \text{ and }
\tilde{\mathscr H}_\ast^\L = \tilde H_\ast (\tilde{\mathscr C}_\bullet^\L),
\]
called the {\it reduced} ({\it Lipschitz}) {\it measure homologies}.
They are also represented as 
\begin{align*}
\tilde{\mathscr H}_\ast(X) &= \ker (\mathscr H_\ast(X) \to \mathscr H_\ast(\{\ast\})) \\
\tilde{\mathscr H}_\ast^\L(Y) &= \ker (\mathscr H_\ast^\L(Y) \to \mathscr H_\ast^\L(\{\ast\}))
\end{align*}
for a topological space $X$ and a metric space $Y$, where the maps between the homologies are induced by the trivial maps $X \to \{\ast\}$ and $Y \to \{\ast\}$ to a one-point space. 
From the definition, the following trivially holds. 

\begin{lemma}\label{lem:singleton}
For a one-point space $\{\ast\}$, $\tilde{\mathscr H}_\ast(\{\ast\}) = 0$ and $\tilde{\mathscr H}_\ast^\L (\{\ast\}) = 0$.
\end{lemma}

The following statement is an expression of a statement proved in \cite{H} and \cite{Z} in other words.
\begin{theorem}[\cite{H}, \cite{Z}] \label{thm:LT mh}
Let $U$ be a subset of a topological space $V$ which is contractible in $V$. 
Namely, there is a continuous map $h : U \times [0,1] \to V$ such that $h_0$ is the inclusion $U \hookrightarrow V$ and $h_1$ is a constant map.
Then, the inclusion $h_0$ induces the zero map $\tilde{\mathscr H}_k(U) \to \tilde{\mathscr H}_k(V)$ for every $k \ge 0$.

In particular, for a locally contractible topological space $X$, 
$\tilde{\mathscr H}_k : \mathsf O(X) \to \mathsf{Vect}$ is a locally trivial precosheaf for each $k \ge 0$.
\end{theorem}

We prove a statement similar to Theorem \ref{thm:LT mh} as follows. 

\begin{theorem} \label{thm:LT Lmh}
Let $U \subset V$ be Lipschitz contractible in a metric space $V$ in the sense that  
there is a Lipschitz map $h : U \times [0,1] \to V$ such that $h_0$ is the inclusion $U \hookrightarrow V$ and $h_1$ is a constant map. 
Then, the inclusion $h_0$ induces the zero map $\tilde{\mathscr H}_k^\L(U) \to \tilde{\mathscr H}_k^\L(V)$ for every $k \ge 0$.

In particular, if $X$ is a locally Lipschitz contractible metric space, then $\tilde{\mathscr H}_k^\L : \mathsf O(X) \to \mathsf{Vect}$ is a locally trivial precosheaf for each $k \ge 0$.
\end{theorem}

To prove Theorem \ref{thm:LT Lmh}, we review an outline of the proof of Theorem \ref{thm:LT mh}.
For a regular $k$-simplex $\T^k$, the prism $\T^k \times I$ has a standard decomposition into $(k+1)$-simplices $P_i$ ($i=0, \dots, k$): 
\[
\T^k \times I = \bigcup_{i=0}^k P_i.
\]
Here, we assume that all the simplicies $P_i$ have positive orientation associated to the orientation of $\T^k$. 
Let $\sigma \in C(\T^k, X)$ be a singular simplex in a topological space $X$.
We define $\sigma \times \mathrm{id}_I : \T^k \times I \to X \times I$ by $(\sigma \times \mathrm{id}_I)(s,t) = (\sigma(s),t)$ for $s \in \T^k$ and $t \in I$.
Setting $P_i \sigma := (\sigma \times \mathrm{id}_I) |_{P_i}$, the prism decomposition of $\sigma$ is given by 
\[
P \sigma = \sum_{i=0}^k P_i \sigma.
\]
Its linear extension $P : C_k(X) \to C_{k+1}(X \times I)$ is a chain homotopy between $i_0{}_\#$ and $i_1{}_\#$, i.e., 
it satisfies
\begin{equation} \label{eq:P}
\partial P - P \partial = i_0{}_\# - i_1{}_\#
\end{equation}
where $C_\bullet(\,\cdot\,)$ denotes the real singular chain complex.
Here, the map $i_t : X \to X \times I$ is given by $i_t(x) = (x,t)$ for $t \in I$.

For each $i = 0,\dots,k$, 
the map $P_i : C(\T^k,X) \to C(\T^{k+1},X \times I)$ is continuous in the compact-open topology. 
So, it induces a map $P_i{}_\# : \mathscr C_k(X) \to \mathscr C_{k+1}(X \times I)$. 
Then, the sum $P_\# = \sum_{i=0}^k P_i{}_\#$ satisfies
\begin{equation} \label{eq:P m}
\partial P_\# - P_\# \partial = i_0{}_\# - i_1{}_\#.
\end{equation}
The relation \eqref{eq:P m} can be verified by a similar way to verify the relation \eqref{eq:P}.
This implies Theorem \ref{thm:LT mh}, due to Lemma \ref{lem:singleton}.

\begin{proof}[Proof of Theorem \ref{thm:LT Lmh}]
Let $X$ be a metric space.
If $\sigma : \T^k \to X$ is Lipschitz, then so is $P_i \sigma$.
Further, the map 
\[
P_i : \Lip(\T^k,X) \to \Lip(\T^{k+1},X \times I)
\]
is continuous in the our topology, due to $(4)$ and $(3)$ of Theorem \ref{thm:top Lip}.
Then, the map
\[
P_\# = \sum_{i=0}^k P_i{}_\#
: \mathscr C_k^\L (X) \to \mathscr C_{k+1}^\L (X \times I)
\]
is verified to be a chain homotopy 
between $i_0{}_\#$ and $i_1{}_\#$ 
by a way similar to verify that \eqref{eq:P m} is a chain homotopy.
This and Lemma \ref{lem:singleton} imply the conclusion of Theorem \ref{thm:LT Lmh}.
\end{proof}

An augmentation $\tilde \partial_0$ of the current chain complex $\N_\bullet^\cpt(X)$ was considered in \cite{RS} and was defined by 
\[
\tilde \partial_0 : \N_0^\cpt(X) \ni T \mapsto T(1) \in \mathbb R.
\]
Actually, it satisfies $\tilde \partial_0 \partial_1 T = 0$ for $T \in \N_1^\cpt(X)$.

\begin{lemma} \label{lem:commute1}
Let $X$ be a metric space.
The following diagram 
\[
\xymatrix{
\mathscr C_0^\L(X) \ar[rr]^{\mu \mapsto T^\mu} \ar[dr]_{\tilde \partial_0} &&\N_0^\cpt(X) \ar[dl]^{\tilde \partial_0} \\
& \mathbb R &
}
\]
commutes.
\end{lemma}
\begin{proof}
Let $\mu \in \mathscr C_0^\L(X)$. Then, we have
\[
\tilde \partial_0 T^\mu = T^\mu(1) = \mu(X) = \tilde \partial_0 \mu.
\]
This completes the proof. 
\end{proof}

We already know the following
\begin{theorem}[\cite{Mi}] \label{thm:LT ch}
If $X$ is an locally Lipschitz contractible metric space, then the precosheaf $\mathbf H_k$ on $X$ is locally trivial for every $k \ge 1$.
\end{theorem}

By summarizing above preparations,  
we obtain
\begin{proof}[Proof of Theorem \ref{thm:main thm}]
This follows from Lemmas \ref{lem:0-th} and \ref{lem:commute1}, Theorems \ref{thm:LT ch}, \ref{thm:LT mh}, \ref{thm:LT Lmh} and \ref{thm:coincidence} and Corollary \ref{cor:mc is cosheaf}. 
\end{proof}

Let us prove Corollaries \ref{cor:top inv}--\ref{cor:alex}. 

\begin{proof}[Proof of Corollary \ref{cor:top inv}]
This follows from Theorem \ref{thm:main thm0} and results in \cite{H} and \cite{Z}. 
\end{proof}

\begin{proof}[Proof of Corollary \ref{cor:Haus dim}]
Let $X$ be a metric space of Hausdorff dimension $< n$, for a nonnegative integer $n$.
Then, due to \cite[Theorem 3.9]{AK}, $\mathbf N_k(X) = 0$ for every integer $k \ge n$.
Therefore, we have $\mathbf H_k(X) = 0$ for $k \ge n$. 
Hence, if $X$ is locally Lipschitz contractible, then by Theorem \ref{thm:main thm0}, we obtain $\mathscr H_k(X) = 0$ for all $k \ge n$.
\end{proof}

\begin{proof}[Proof of Corollary \ref{cor:sh}]
Let $X$ be an LLC metric space and $Y$ a finite CW-complex. 
Suppose that there is a homotopy equivalence $h : X \to Y$. 
By \cite{H} and \cite{Z}, there is a commutative diagram consisting of isomorphisms:
\[
\xymatrix{
H_\ast(X) \ar[r] \ar[d]_{h_\ast} & \mathscr H_\ast(X) \ar[d]^{h_\ast} \\
H_\ast(Y) \ar[r] & \mathscr H_\ast(Y).
}
\]
Here, $H_\ast$ is the usual singular real homology.
Due to Theorem \ref{thm:main thm}, we have $\mathscr H_\ast(X) \cong \mathbf H_\ast(X)$. 
Further, if $Y$ is a finite CW-complex, then $\dim H_\ast(Y)< \infty$. 
Therefore, we obtain the conclusion.
\end{proof}

\begin{proof}[Proof of Corollary \ref{cor:alex}]
Let $X$ be an $n$-dimensional compact orientable Alexandrov space without boudnary as in the assumption. 
By \cite{Y}, we have $H_n(X) \cong \mathbb R$. 
Since $X$ has the homotopy type of a CW-complex, by Corollary \ref{cor:sh} and Theorem \ref{thm:main thm0}, we obtain $\mathbf H_n(X) \cong \mathscr H_n(X) \cong \mathbb R$.
This completes the proof.
\end{proof}


\begin{theorem} \label{thm:main thm:rel}
On the category of all pairs of metric spaces and all locally Lipschitz maps, there are natural transformations $\mathscr C_\bullet \leftarrow \mathscr C_\bullet^\L \to \N_\bullet^\cpt$. 
If they are restricted to the category of all pairs of locally Lipschitz contractible metric spaces and all locally Lipschitz maps, then they induce isomorphisms $\mathscr H_\ast \leftarrow \mathscr H_\ast^\L \to \mathbf H_\ast$ between the homologies. 
\end{theorem}
\begin{proof}
This follows from Theorem \ref{thm:main thm} together with a standard argument of homological algebra (the five lemma and the snake lemma).
\end{proof}

\begin{remark} \upshape \label{rem:PS}
In \cite{PS1} and \cite{PS2}, Paolini and Stepanov thoroughly researched one-dimensional normal currents in arbitrary metric spaces. 
For a metric space $E$, they provided a pseudo-distance function $d_\Theta$ on $\Lip([0,1],E)$ which is related to the uniform distance. 
The quotient space of $\Lip([0,1],E)$ under the relation $d_\Theta = 0$ was denoted by $\Theta(E)$.
By the definition of $d_\Theta$, it is known that if $d_\Theta(\theta, \theta') = 0$, then $[\theta] = [\theta']$ holds.
Their results say that {\it every} normal one-dimensional current in {\it every} metric space $E$ is represented by the integral of $[\theta]$ in a positive Borel measure on $\Theta(E)$ (and in a positive Borel measure on $C([0,1],E)$ concentrated on $\Lip([0,1],E)$), vice versa. 
See \cite{PS1} and \cite{PS2} for more details.
In contrast to their results, our Theorem \ref{thm:main thm} says that every normal current {\it cycle} (i.e., its boundary is zero) of {\it every dimension}, in every LLC metric space is represented by an integral of $[\theta]$. 
\end{remark}

\section{A topology on the space of bounded Lipschitz maps} \label{sec:top Lip}
In this section, we provide a topology on the set of all bounded Lipschitz maps between general metric spaces which satisfies the properties as in Theorem \ref{thm:top Lip} and additional properties.

\subsection{Banach space target} \label{subsec:BT}
In this subsection, we define a topology on the set of all bounded Lipschitz maps from a metric space to a Banach space. 

Let $B$ be a Banach space over real numbers and $X$ a metric space. 
Let $\Lipb(X,B)$ be the set of all bounded Lipschitz maps from $X$ to $B$. 
Then, it is a real vector space associated to the standard addition and scalar multiplication operators. 
We consider a norm on $\Lipb(X,B)$ defined by 
\[
\|f\| = \|f\|_\infty + \Lip(f),
\]
where $\|f\|_\infty$ is the supremum norm. 
The set $\Lipb(X,B)$ equipped with the topology induced by $\|\cdot\|$ is denoted by $\BT(X,B)$. 
Here, the symbol $\mathrm{BT}$ indicates ``Banach space Target''. 
This topology has the following fundamental properties. 

\begin{proposition} \label{prop:BT}
Let $X$ and $X'$ be metric spaces and $B$ and $B'$ Banach spaces. 
Then, the following holds. 
\begin{itemize}
\item[(a)]
Let $f_j, f \in \BT(X,B)$ with $j \in \mathbb N$. 
Then, $f_j$ converges to $f$ in the topology of $\BT(X,B)$ if and only if $\|f_j - f\|_\infty \to 0$ and $\Lip(f_j - f) \to 0$ as $j \to \infty$; 
\item[(b)]
Let $\phi : B \to B'$ be a continuous linear map. 
Then, the map $\phi_\# : \BT(X,B) \to \BT(X,B')$ defined by the composition $\phi_\# (f) = \phi \circ f$, is a continuous linear map. 
In addition, if $\phi$ is injective and its image $\phi(B)$ is closed in $B'$, then $\phi_\#$ is a topological embedding. 
\item[(c)]
Let $\psi : X \to X'$ be a Lipschitz map. 
Then, the composition $\psi^\# : \BT(X',B) \to \BT(X,B); f \mapsto f \circ \psi$ is a continuous linear map.
\item[(d)]
The map 
\[
\BT(X,B) \times \BT(X',B') \to \BT(X \times X',B \times B')
\]
defined by $(f,g) \mapsto f \times g$ is a continuous linear map. 
Here, $f \times g$ is given by $(f \times g)(x,x') = (f(x),g(x'))$ for $x \in X$ and $x' \in X'$.
\item[(e)]
Let $\{\ast\}$ be a single-point set. 
Then, the map $\BT(\{\ast\}, B) \ni f \mapsto f(\ast) \in B$ is a linear homeomorphism.
\end{itemize}
\end{proposition}

\begin{proof}
The property (d) follows from (b) and (c). 
The properties (a), (b), (c) and (e) are easily proved. 
We give a proof of (b) for the convenience. 
Let $\phi : B \to B'$ be a continuous linear map between Banach spaces $B$ and $B'$, and $X$ a metric space. 
It is trivial that $\phi_\#$ is linear. 
For any $f \in \Lipb(X,B)$, we have 
\[
\|\phi \circ f\|_\infty \le \|\phi\|_\mathrm{op} \|f\|_\infty \text{ and } \Lip(\phi \circ f) \le \|\phi\|_\mathrm{op} \Lip(f),
\]
where $\|\phi\|_\mathrm{op}$ is the operator norm of $\phi$. 
Hence, $\phi_\#$ is continuous. 
The second statement of (b) follows from the inverse mapping theorem. 
\end{proof}

\subsection{Double dual of metric spaces} \label{subsec:double dual}

Every metric space admits an isometric embedding into a Banach space. 
Several such constructions are known. 
Among them, we choose the following way.
Let $X$ be a metric space and $x_0 \in X$. 
Let $\Lip_{x_0}(X)$ denote the Banach space of all real-valued Lipschitz functions on $X$ vanishing at $x_0$, equipped with the norm taking the smallest Lipschitz constant.
We denote by $X_{x_0}^{\ast\ast}$ the continuous dual of it.
Then, a map $\delta : X \to X_{x_0}^{\ast\ast}$ defined by 
\begin{equation} \label{eq:delta}
\delta_x (f) = f (x)
\end{equation}
for all $x \in X$ and $f \in \Lip_{x_0}(X)$, can be easily verified to be an isometric embedding. 
Obviously, $\delta_{x_0}$ becomes the zero vector in $X_{x_0}^{\ast\ast}$.
In this subsection, we observe fundamental properties of this construction. 
In particular, we see that the construction $(X,x_0) \mapsto X_{x_0}^{\ast\ast}$ satisfies a covariant functorial property (Proposition \ref{prop:covariant}). 

\begin{remark}\upshape
The closed linear span of $\{\delta_x \in X_{x_0}^{\ast\ast} \mid x \in X\}$ is called the (Lipschitz) free Banach space or the Arens-Eells space associated to $(X, x_0)$ (see e.g. \cite{W}).
\end{remark}

For a Lipschitz map $\phi : X \to Y$ between metric spaces and $x_0 \in X$, we define 
\[
\phi_\# : X_{x_0}^{\ast\ast} \to Y_{\phi(x_0)}^{\ast\ast}
\]
by 
\[
(\phi_\# \mu) (f) = \mu(f \circ \phi)
\]
for all $\mu \in X_{x_0}^{\ast\ast}$ and $f \in \Lip_{\phi(x_0)}(Y)$.

\begin{lemma} \label{lem:commute}
For a Lipschitz map $\phi : X \to Y$ and $x_0 \in X$, 
\[
\phi_\# \circ \delta = \delta \circ \phi
\]
holds. 
\end{lemma}
\begin{proof}
Let $x \in X$ and $f \in \Lip_{x_0}(X)$. 
Then, $\phi_\# \delta_x (f) = \delta_x (f \circ \phi) = f (\phi(x)) = \delta_{\phi(x)} (f)$.
Namely, $\phi_\# \circ \delta = \delta \circ \phi$ holds.
\end{proof}

\begin{proposition} 
\label{prop:covariant}
Let $\phi : X \to Y$ and $x_0 \in X$ as above. 
Then, $\phi_\# : X_{x_0}^{\ast\ast} \to Y_{\phi(x_0)}^{\ast\ast}$ is a bounded linear operator with operator norm $\Lip(\phi)$.
\end{proposition}
\begin{proof}
It is trivial that $\phi_\#$ is linear.
For $f \in \Lip_{\phi(x_0)}(Y)$, 
\[
|\phi_\# \mu (f)| = |\mu (f \circ \phi)| \le \| \mu \| \Lip(f) \Lip(\phi).
\]
Hence, $\|\phi_\# \mu\| \le \|\mu\| \Lip(\phi)$.
It implies $\|\phi_\# \| \le \Lip(\phi)$.
By Lemma \ref{lem:commute}, for every $x \neq y \in X$, we have 
\[
\frac{\|\phi_\# \delta_x - \phi_\# \delta_y \|}{\|\delta_x - \delta_y\|} = \frac{d(\phi(x),\phi(y))}{d(x,y)}.
\]
Since this value can be taken to be arbitrary close to $\Lip(\phi)$, we obtain $\|\phi_\#\| = \Lip(\phi)$. 
\end{proof}

\begin{corollary} \label{cor:bi-Lip iso}
If $\phi : X \to Y$ is a bi-Lipschitz homeomorphism and $x_0 \in X$, then $\phi_\# : X_{x_0}^{\ast\ast} \to Y_{\phi(x_0)}^{\ast\ast}$ is a linear homeomorphism.
\end{corollary}

\begin{lemma}\label{lem:inclusion}
If $X$ is a subset of a metric space $Y$, then the inclusion $i : X \to Y$ induces an isometric linear embedding $i_\# : X_{x_0}^{\ast\ast} \to Y_{x_0}^{\ast\ast}$ for every $x_0 \in X$.
\end{lemma}
\begin{proof}
Let $x_0 \in X$ be fixed.
The inclusion $i : X \hookrightarrow Y$ induces a linear map $i^\# : \Lip_{x_0}(Y) \to \Lip_{x_0}(X)$ given by $f \mapsto f \circ i = f |_X$ for $f \in \Lip_{x_0}(Y)$.
Since $\Lip(f |_X) \le \Lip(f)$ for all $f \in \Lip_{x_0}(Y)$, the operator norm of $i^\#$ is not greater than $1$.
Due to the McShane-Whitney Lipschitz extension theorem, the map $i^\#$ is surjective.
Dually, the bounded linear operator $i_\# : X_{x_0}^{\ast\ast} \to Y_{x_0}^{\ast\ast}$ is injective and its operator norm $\le 1$.
We prove that $\|i_\# \mu\| = \|\mu\|$ for every $\mu \in X_{x_0}^{\ast\ast}$.
We assume that there exist $\epsilon > 0$ and $\mu \in X_{x_0}^{\ast\ast}$ such that $\|i_\# \mu\| \le (1-\epsilon) \|\mu\|$.
We may assume that $\|\mu\| = 1$. 
Hence, we have
\[
|\mu (f |_X)| \le (1-\epsilon) \Lip(f)
\]
for every $f \in \Lip_{x_0}(Y)$.
Again, due to the McShane-Whitney extension theorem, for every $g \in \Lip_{x_0}(X)$, there is $f \in \Lip_{x_0}(Y)$ such that $f|_X = g$ and $\Lip(f) = \Lip(g)$. 
This yields, 
\[
|\mu(g)| \le (1-\epsilon) \Lip(g)
\]
for every $g \in \Lip_{x_0}(X)$.
It implies $\|\mu\| \le 1 - \epsilon$ which contradicts to the assumption $\|\mu\| = 1$.
Therefore, $\|i_\# \mu\| = \|\mu\|$ holds for every $\mu$.
This completes the proof.
\end{proof}

\begin{corollary} \label{cor:bi-Lip}
If $\phi : X \to Y$ be a bi-Lipschitz embedding, then $\phi_\# : X_{x_0}^{\ast\ast} \to Y_{\phi(x_0)}^{\ast\ast}$ is an injective bounded linear map having closed image. 
\end{corollary}
\begin{proof}
This follows from Lemma \ref{lem:inclusion} and Corollary \ref{cor:bi-Lip iso}.
\end{proof}

For $x_0, x_1 \in X$, a canonical isometric isomorphism 
\[
\Lip_{x_0}(X) \to \Lip_{x_1}(X)
\]
is defined by $f \mapsto f - f(x_1)$. 
It implies an isometric isomorphism
\begin{equation} \label{eq:trans}
X_{x_1}^{\ast\ast} \to X_{x_0}^{\ast\ast}.
\end{equation}
Namely, for $\mu \in X_{x_1}^{\ast\ast}$, the map \eqref{eq:trans} assigns an element $\mu' \in X_{x_0}^{\ast\ast}$ defined by 
\[
\mu'(f) = \mu(f - f(x_1))
\]
for all $f \in \Lip_{x_0}(X)$.

\begin{lemma} \label{lem:product DD}
Let $X$ and $Y$ be metric spaces with $x_0 \in X$ and $y_0 \in Y$. 
Then, a map 
\[
\Phi : X_{x_0}^{\ast\ast} \times Y_{y_0}^{\ast\ast} \to (X \times Y)_{(x_0,y_0)}^{\ast\ast}
\]
given by 
\[
\Phi(\mu,\nu)(h) = \mu(h(\cdot,y_0)) + \nu(h(x_0,\cdot))
\]
for $(\mu,\nu) \in X_{x_0}^{\ast\ast} \times Y_{y_0}^{\ast\ast}$ and $h \in \Lip_{(x_0,y_0)}(X \times Y)$, 
is a continuous linear map. 
\end{lemma}
\begin{proof}
Let $(\mu,\nu) \in X_{x_0}^{\ast\ast} \times Y_{y_0}^{\ast\ast}$ and $h \in \Lip_{(x_0,y_0)}(X \times Y)$. 
Then, we have 
\begin{align*}
|\Phi(\mu,\nu)(h)| &\le |\mu(h(\cdot,y_0))| + |\nu(h(x_0,\cdot))| \\
&\le \Lip(h) \{\|\mu\| + \|\nu\|\}.
\end{align*}
Therefore, $\Phi$ is a bounded linear map.
\end{proof}

One can also prove that the map $\Phi$ in Lemma \ref{lem:product DD} is injective and has the closed image. 

When $V$ is a Banach space, 
let us compare the space $V_0^{\ast\ast} = \Lip_0(V)^\ast$ with the usual continuous double dual $V^{\ast\ast}$, where $0$ is the zero vector in $V$.
Since the operator norm of a linear map is no other than its Lipschitz constant, the continuous dual $V^\ast$ of $V$ is contained in $\Lip_0(V)$ as a closed subspace:
\[
V^\ast \subset \Lip_0(V).
\]
Dually, we obtain a surjective bounded linear operator 
\begin{equation} \label{eq:usual double}
r : V_0^{\ast\ast} \twoheadrightarrow V^{\ast\ast}
\end{equation}
assigning the restriction $f |_{V^\ast}$ to $V^\ast$ for each $f \in V_0^{\ast\ast}$.
Here, the surjectivity of $r$ follows from the Hahn-Banach theorem.
Then, we obtain a canonical map 
\[
\bar \delta : V \to V^{\ast\ast}
\]
defied by 
\[
\bar \delta_x = \delta_x |_{V^\ast}
\]
for $x \in V$. 
This map $\bar \delta$ is no other than the usual canonical isometric linear embedding of $V$ into $V^{\ast\ast}$ defined by the evaluation.

\subsection{Metric space target} \label{subsec:MT}
Let $A$ and $X$ denote metric spaces. 
We equip $\Lipb(A,X)$ with a topology satisfying the desired properties stated in Theorem \ref{thm:top Lip} and additional properties.

We fix $x_0 \in X$ and an isometric embedding $\delta : X \to X_{x_0}^{\ast\ast}$ defined in \eqref{eq:delta}. 
It implies an injection
\[
\delta_\# : \Lipb(A,X) \to \Lipb(A,X_{x_0}^{\ast\ast})
\]
given by $\delta_\# f = \delta_{f}$ for all $f \in \Lipb(A,X)$.
We endow $\Lipb(A,X)$ with the coarsest topology in which the map $\delta_\# : \Lipb(A,X) \to \BT(A,X_{x_0}^{\ast\ast})$ is continuous.

\begin{lemma}
The topology on $\Lipb(A,X)$ given as above, is independent on the choice of base point $x_0 \in X$.
\end{lemma}
\begin{proof}
Let us fix $x_0, x_1 \in X$.
The isometric embeddings 
\[
\delta^0 : X \to X_{x_0}^{\ast\ast} \text{ and } \delta^1 : X \to X_{x_1}^{\ast\ast}
\]
are given by the same maps $\delta^0 = \delta^1 = \delta$.
Let $\Psi : X_{x_1}^{\ast\ast} \to X_{x_0}^{\ast\ast}$ denote a canonical map defined in \eqref{eq:trans}.
Then, 
\[
\Psi_\# : \BT(A,X_{x_1}^{\ast\ast}) \to \BT(A,X_{x_0}^{\ast\ast})
\]
is homeomorphic due to Proposition \ref{prop:BT} (b) and Corollary \ref{cor:bi-Lip iso}.
For $f \in \Lipb(A,X)$ and $h \in \Lip_{x_0}(X)$, we have
\begin{align*}
\Psi_\# \delta^1_\# f (h) &= \Psi \circ \delta_f (h) = \delta_f(h - h(x_1)) = h \circ f(\cdot) - h(x_1) \text{ and}, \\
\delta^0_\# f (h) &= \delta_f(h) = h \circ f (\cdot).
\end{align*}
Therefore, $\Psi_\# \delta^1_\# - \delta^0_\# : \Lipb(A,X) \to \BT(A,X_{x_0}^{\ast\ast})$ is a constant map. 
Since $\Psi_\#$ is homeomorphic, $\delta^1_\#$ is continuous if and only if so is $\delta^0_\#$. 
This completes the proof.
\end{proof}

Let us denote by $\MT(A,X)$ the space $\Lipb(A,X)$ with the topology induced by $\delta_\#$.
Here, the symbol $\mathrm{MT}$ indicates ``Metric space Target''. 

\begin{proposition} \label{prop:met1}
The space $\MT(A,X)$ is metrizable. 
If a sequence $f_j$ converges to $f$ in $\MT(A,X)$, then $\Lip(f_j) \to \Lip(f)$ and $f_j$ converges to $f$ uniformly as $j \to \infty$.
In particular, the statements $(0)$ and $(1)$ of Theorem \ref{thm:top Lip} hold. 
\end{proposition}
\begin{proof}
This follows from the definition of the topology.
\end{proof}

\begin{proposition}\label{prop:met2}
The statement $(3)$ of Theorem \ref{thm:top Lip} holds. 
Namely, for $\phi : A \to A'$ a Lipschitz map between metric spaces, 
the map $\phi^\# : \MT(A',X) \to \MT(A,X)$ defined by $\phi^\# f = f \circ \phi$ is continuous. 
\end{proposition}
\begin{proof}
Let us fix $x_0 \in X$. 
The following diagram 
\[
\begin{CD}
\MT(A',X) @> \delta_\# >> \BT(A',X_{x_0}^{\ast\ast}) \\
@V \phi^\# VV @VV \phi^\# V \\
\MT(A,X) @> \delta_\# >> \BT(A,X_{x_0}^{\ast\ast}) 
\end{CD}
\]
consisting of canonical maps, commutes. 
From Proposition \ref{prop:BT} (c), we obtain the conclusion.
\end{proof}

\begin{proposition} \label{prop:met3}
The statement $(2)$ of Theorem \ref{thm:top Lip} holds. 
Namely, for a Lipschitz map $\phi : X \to Y$ between metric spaces, 
the map $\phi_\# : \MT(A,X) \to \MT(A,Y); f \mapsto \phi \circ f$ is continuous. 
Further, if $\phi$ is a bi-Lipschitz embedding, then $\phi_\#$ is a topological embedding.
\end{proposition}
\begin{proof}
By Lemma \ref{lem:commute},  
the following diagram 
\[
\begin{CD}
\MT(A,X) @> \delta_\# >> \BT(A,X_{x_0}^{\ast\ast}) \\
@V \phi_\# VV @VV \phi_\#{}_\# V \\
\MT(A,Y) @> \delta_\# >> \BT(A,Y_{\phi(x_0)}^{\ast\ast}) 
\end{CD}
\]
consisting of canonical maps, commutes. 
By proposition \ref{prop:covariant}, the map $\phi_\# : X_{x_0}^{\ast\ast} \to Y_{\phi(x_0)}^{\ast\ast}$ is a bounded linear map.
Hence, due to Proposition \ref{prop:BT} (b), the map $\phi_\#{}_\# : \BT(A,X_{x_0}^{\ast\ast}) \to \BT(A,Y_{\phi(x_0)}^{\ast\ast})$ is continuous.
Therefore, $\phi_\# : \MT(A,X) \to \MT(A,Y)$ is continuous. 
The second statement follows from Proposition \ref{prop:BT} (b) and Corollary \ref{cor:bi-Lip}. 
\end{proof}

\begin{remark} \upshape
To prove Proposition \ref{prop:met3}, it is important the functorial property of the correspondence $(X,x_0) \mapsto X_{x_0}^{\ast\ast}$ (Proposition \ref{prop:covariant}) and Lemma \ref{lem:commute}.

For instance, Kuratowski embedding, which is a famous isomtric embedding into a Banach space, does not have the functorial property. 
\end{remark}

\begin{proposition}\label{prop:met4}
The statement $(4)$ of Theorem \ref{thm:top Lip} holds. 
Namely, for metric spaces $A,B,X$ and $Y$, 
the canonical map 
\[
\MT(A,X) \times \MT(B,Y) \to \MT(A \times B, X \times Y)
\]
is continuous. 
\end{proposition}
\begin{proof}
Let us fix $x_0 \in X$ and $y_0 \in Y$.
Let us consider the following commutative diagram
\[
\xymatrix{
\MT(A,X) \times \MT(B,Y) \ar[r] \ar[d] & \MT(A \times B, X \times Y) \ar[dd] \\
\BT(A,X_{x_0}^{\ast\ast}) \times \BT(B,Y_{y_0}^{\ast\ast}) \ar[d] & \\
\BT(A \times B, X_{x_0}^{\ast\ast} \times Y_{x_0}^{\ast\ast}) \ar[r] & \BT(A \times B, (X \times Y)_{(x_0,y_0)}^{\ast\ast})
}
\]
consisting of canonical maps.
Since the right downward arrow is a topological embedding from the definition, 
the top rightward arrow is continuous if and only if the composition of the left two downward arrows and the bottom rightward arrow is continuous. 
It follows from Lemma \ref{lem:product DD} and Proposition \ref{prop:BT} (d).
\end{proof}

\begin{proposition}
The statement $(5)$ of Theorem \ref{thm:top Lip} holds. 
Namely, for a singleton set $\{\ast\}$ and a metric space $X$, the canonical map $X \to \MT(\{\ast\}, X)$ is homeomorphic.
\end{proposition}
\begin{proof}
This follows from Proposition \ref{prop:BT} (e) and the definition of the topology.
\end{proof}

When $V$ is a Banach space and $Z$ is a metric space, we compare the topologies of $\MT(Z,V)$ and $\BT(Z,V)$.

\begin{proposition}\label{prop:TVS}
Let $V$ be a Banach space and $Z$ a metric space. 
Then, the topologies on $\MT(Z,V)$ and $\BT(Z,V)$ coincide with each other.
\end{proposition}
\begin{proof}
We first prove that the identity $\mathrm{id} : \MT(Z,V) \to \BT(Z,V)$ is continuous. 
Let us consider the following commutative diagram
\[
\begin{CD}
\MT(Z,V) @> \mathrm{id} >> \BT(Z,V) \\
@V \delta_\# VV @VV \overline{\delta}_\# V \\
\BT(Z,V_0^{\ast\ast}) @> r_\# >> \BT(Z,V^{\ast\ast}).
\end{CD}
\]
Here, the bottom rightward arrow $r_\#$ is continuous, because it is induced by the bounded linear map 
$r: V_0^{\ast\ast} \twoheadrightarrow V^{\ast\ast}$ given in \eqref{eq:usual double}. 
Since $\bar \delta : V \to V^{\ast\ast}$ is an isometric linear embedding, 
the induced map $\bar \delta_\#$ is a topological embedding, due to Proposition \ref{prop:BT} (b).
It follows from the continuity of $r_\# \circ \delta_\#$ that $\mathrm{id} : \MT(Z,V) \to \BT(Z,V)$ is continuous.

Every neighborhood at $0$ in $\MT(Z,V)$ is generated by sets of form
\begin{align*}
\delta_\#^{-1}(\{g \in \Lipb(Z,V_0^{\ast\ast}) \mid \|g\|_\infty < \rho \text{ and } \Lip(g) < \ell \}) 
\end{align*}
for $\rho, \ell > 0$.
Since $\delta$ is the isometric embedding, these sets are equal to 
\begin{align*}
\{f \in \Lipb(Z,V) \mid \|f\|_\infty < \rho \text{ and } \Lip(f) < \ell \}
\end{align*}
which are also open neighborhoods of $0$ in $\BT(Z,V)$. 
Hence, the identity $\mathrm{id} : \BT(Z,V) \to \MT(Z,V)$ is continuous at $0$. 
Since $\MT(Z,V)$ is a topological group by Proposition \ref{prop:TG}, the group homomorphism $\mathrm{id} : \BT(Z,V) \to \MT(Z,V)$ is continuous on the whole set. 
\end{proof}

As a corollary to Propositions \ref{prop:TVS} and \ref{prop:met3}, we obtain 
\begin{corollary} \label{cor:emb}
Let $Z$ and $X$ be metric space. 
Let $\phi : X \to V$ be a bi-Lipschitz embedding into a Banach space $V$.
Then, the map $\phi_\# : \MT(Z,X) \to \BT(Z,V)$ is a topological embedding.

That is, the topology on $\MT(Z,X)$ coincides with the subspace topology by regarding $\Lipb(Z,X)$ as a subspace of $\BT(Z,V)$ via the injection $\phi_\#$.
\end{corollary}

Finally, we remark a relation between our topology on the space of Lipschitz maps and the $C^1$-topology on the space of smooth maps, when a domain and a target are smooth compact Riemannian manifolds. 

Let us denote by $M$ and $N$ compact smooth manifolds, where $N$ has no boundary and $M$ possibly has piecewise smooth boundary.
The set of all $C^1$-maps from $M$ to $N$ is denoted by $C^1(M,N)$. 
We fix Riemannian metrics on $M$ and $N$, and regard them as metric spaces associated to the Riemannian metrics. 
Since $M$ is compact, 
$C^1(M,N)$ is a subset of $\Lip(M,N)$. 
The following gives a characterization of the $C^1$-topology on $C^1(M,N)$ in terms of the topology of $\MT(M,N)$:

\begin{proposition} \label{prop:C1}
Let $M$ and $N$ as above.
Then, the relative topology on $C^1(M,N)$ as a subset of $\MT(M,N)$ coincides with the $C^1$-topology on it.
\end{proposition}
\begin{proof}
Let us take a Whitney smooth embedding $\Phi : N \to \mathbb R^K$ into a Euclidean space $\mathbb R^K$ for a large $K \ge \dim N$. 
Then, $\Phi$ is also a bi-Lipschitz embedding, since $N$ is compact. 
We have the following commutative diagram 
\[
\begin{CD}
C^1(M,N) @> \subset >> \MT(M,N) \\
@V \Phi_\# VV @VV \Phi_\# V \\
C^1(M,\mathbb R^K) @> \subset >> \MT(M,\mathbb R^K). 
\end{CD}
\]
Since both two $\Phi_\#$ in this diagram are topological embeddings, if $C^1(M,\mathbb R^K) \subset \MT(M,\mathbb R^K)$ is a topological embedding, then so is $C^1(M,N) \subset \MT(M,N)$.
Due to Proposition \ref{prop:TVS}, $\BT(M,\mathbb R^K)$ and $\MT(M,\mathbb R^K)$ are same as topological spaces.
Let us take a sequence $f_j$ and an element $f$ in $C^1(M,\mathbb R^K)$. 
Since $\Lip(g) = \|\nabla g\|_\infty$ for any $g \in C^1(M,\mathbb R^K)$, the sequence $f_j$ converges to $f$ in the topology of $\MT(M,\mathbb R^K) = \BT(M,\mathbb R^K)$ if and only if it converges to $f$ in the $C^1$-topology. 
This completes the proof.
\end{proof}

\end{document}